\documentclass[10pt]{amsart}

\newcommand{\cal}[1]{\mathcal{#1}}

\usepackage{john}
\usepackage{hyperref}
\usepackage{tikz}
\usepackage{pgfplots}

\numberwithin{theorem}{section}
\usepackage{color}
\usepackage{framed}

\makeatletter
\let\@@pmod\pmod
\DeclareRobustCommand{\pmod}{\@ifstar\@pmods\@@pmod}
\def\@pmods#1{\mkern4mu({\operator@font mod}\mkern 6mu#1)}
\makeatother

\newcommand{\NP}{\oper{NP}}

\newcommand{\ve}{\varepsilon}
\newcommand{\cP}{\mathcal P}
\newcommand{\bfloor}[1]{\left\lfloor#1\right\rfloor}

\newcommand{\Fp}{{\mathbf F}_p}

\newtheoremstyle{smallexample}
{}
{ }
{\small\addtolength{\leftskip}{3em}}
{ }
{\bfseries}
{.}
{ }
{\thmname{#1}}

\theoremstyle{smallexample}
\newtheorem*{sex-full}{$\Gamma_0(N)$-level}
\newtheorem*{sex-rhobar}{$\rhobar$-component}
\newtheorem*{sex-either}{Either example}

\usepackage[displaymath, mathlines]{lineno}

\usepackage{array}
\usepackage{multirow}
\usepackage{wrapfig}

\usepackage[all,cmtip]{xy}
\usepackage[normalem]{ulem}
\usepackage{xcolor}
\newcommand\redsout{\bgroup\markoverwith{\textcolor{red}{\rule[0.5ex]{2pt}{1pt}}}\ULon}
\newcommand\bluesout{\bgroup\markoverwith{\textcolor{blue}{\rule[0.5ex]{2pt}{1pt}}}\ULon}

\date{\today}

\newcommand{\rhobar}{\overline{\rho}}
\newcommand{\Zp}{{\mathbf Z}_p}
\newcommand{\Qp}{\Q_p}

\newcommand{\Fpbar}{\overline{\F}_p}

\renewcommand{\ds}{\displaystyle}
\renewcommand{\d}{d}
\newcommand{\dn}{d^{\new}}
\newcommand{\dt}{d_t}
\newcommand{\dnt}{d_t^{\new}}
\newcommand{\ddp}{d_p}
\newcommand{\ddpt}{d_{p,t}}
\newcommand{\Cd}{A}
\newcommand{\Cdn}{B}

\newcommand{\DD}{D}

\newcommand{\Pd}{P_d}
\newcommand{\Qd}{Q_d}
\newcommand{\Pdt}{P_{d_t}}
\newcommand{\Qdt}{Q_{d_t}}
\newcommand{\Pdn}{P_{d^{\new}}}
\newcommand{\Qdn}{Q_{d^{\new}}}
\newcommand{\Pddn}{P_{d+d^{\new}}}
\newcommand{\Qddn}{Q_{d+d^{\new}}}
\newcommand{\Pdp}{P_{d_p}}

\newcommand{\Qdpt}{Q_{\ddpt}}

\newcommand{\Qdp}{Q_{d_p}}
\newcommand{\Qplus}{Q^+}
\newcommand{\Qminus}{Q^-}
\newcommand{\Qf}{Q}
\newcommand{\Qfr}{Q_r}

\DeclareMathOperator{\HZ}{HZ}
\DeclareMathOperator{\LZ}{LZ}
\DeclareMathOperator{\chr}{char}

\newcommand{\tv}[1]{~\text{#1}~}
\newcommand{\rbar}{\overline{r}}

\renewcommand{\and}{\quad\text{and}\quad }

\newenvironment{smallpmatrix}
  {\left(\begin{smallmatrix}}
  {\end{smallmatrix}\right)}

\title{Slopes of modular forms and the ghost conjecture, II}

\author{John Bergdall and Robert Pollack}

\address{John Bergdall\\Department of Mathematics\\ Michigan State University \\ 619 Red Cedar Road \\ East Lansing, MI 48824\\USA}
\email{bergdall.jf@gmail.com}
\urladdr{http://users.math.msu.edu/users/bergdall}

\address{Robert Pollack\\Department of Mathematics and Statistics \\ Boston University \\ 111 Cummington Mall \\ Boston, MA 02215\\USA}
\email{rpollack@math.bu.edu}
\urladdr{http://math.bu.edu/people/rpollack}

\subjclass[2000]{11F33 (11F85)}

\begin{document}
\begin{abstract}
In a previous article, we constructed an entire power series over $p$-adic weight space (the `ghost series') and conjectured, in the $\Gamma_0(N)$-regular case, that this series encodes the slopes of overconvergent modular forms of any $p$-adic weight.  In this paper, we construct `abstract ghost series' which can be associated to various natural subspaces of overconvergent modular forms.  This abstraction allows us to generalize our conjecture to, for example, the case of slopes of overconvergent modular forms with a fixed residual representation that is locally reducible at $p$.  Ample numerical evidence is given for this new conjecture.  Further, we prove that the slopes computed by any abstract ghost series satisfy a distributional result at classical weights (consistent with conjectures of Gouv\^ea) while the slopes form unions of arithmetic progressions at all weights not in $\Z_p$.
\end{abstract}

\maketitle
\setcounter{tocdepth}{1}

\section{Introduction}

In previous work (\cite{BergdallPollack-GhostPaperShort}) we formulated a conjecture, called the ghost conjecture, on the slopes of overconvergent $p$-adic cuspforms. The form of the conjecture is the following. For a fixed prime $p$ and an integer $N \geq 1$, not divisible by $p$, we explicitly define a power series $G(w,t)$ in $\Zp[[w,t]]$ that we call the `ghost series' and which we view as a series in $t$ with coefficients lying in the ring of functions on a fixed component of the $p$-adic weight space. When $p$ is $\Gamma_0(N)$-regular (as defined by Buzzard in  \cite{Buzzard-SlopeQuestions}), the ghost conjecture asserts that the Newton polygon of the specialization of $G$ to a weight $\kappa$ is the same as the Newton polygon of the $U_p$-operator acting on overconvergent $p$-adic cuspforms of level $\Gamma_0(N)$ and weight $\kappa$. Perhaps surprisingly, the construction of $G$ is fairly simple and certainly elementary to describe. It essentially depends only on the dimensions $d_k := \dim S_k(\Gamma_0(N))$ and $d_k^{\new} := \dim S_k(\Gamma_0(Np))^{p-\new}$ of classical spaces of cuspforms as $k$ varies. Thus the construction is amenable to abstraction.

The purpose of this article is twofold. First, we seek to define an `abstract ghost series' depending just on dimension-like functions $d$ and $d^{\new}$. Second, we aim to re-specialize our abstraction in order to formulate a ghost conjecture we believe valid for a fixed modular mod $p$ Galois representation which is further assumed to be reducible upon restriction to a decomposition group at $p$ (the Galois-theoretic interpretation of `$\Gamma_0(N)$-regular').

This abstraction allows us to study consequences of the ghost conjecture in one fell swoop. For instance, the data underlying the definition of an abstract ghost series allows for a natural interpretation of `classical slopes.' In Section \ref{sec:dist}, we prove that these classical slopes satisfy a distribution law that specializes to the `Gouv\^ea distribution' when the ghost series is as in \cite{BergdallPollack-GhostPaperShort} or is attached to a fixed mod $p$ Galois representation $\rhobar$. We further explore the combinatorics of what we called `global halos' in \cite{BergdallPollack-GhostPaperShort} (generalizing Coleman's conjectural `spectral halos'). As intimated in previous work, the presence of halos is a structural feature of the ghost series. We prove, in Section \ref{sec:progressions}, that the slopes of the halos form finite unions of arithmetic progressions. The persistence of this behavior $\rhobar$-by-$\rhobar$ is neither completely surprising, nor does it follow from any of the partial results (e.g.\ \cite{BergdallPollack-FredholmSlopes,LiuXiaoWan-IntegralEigencurves}) on the actual spectral halos.

The proofs of the results described in the previous paragraph provides proofs of the unjustified statements found in \cite{BergdallPollack-GhostPaperShort}. We refer the reader to our previous article for a more robust discussion about slopes of modular forms, historical notes, and what is currently known about the ghost conjecture. 

We end by emphasizing that just as our previous work seems to be a generalization of a conjecture of Buzzard (\cite{Buzzard-SlopeQuestions}), the $\rhobar$-ghost conjecture discussed at the end of this article likely generalizes the conjecture in Clay's Ph.D.\ thesis (\cite{Clay-Slopes}).

The first half of the article (Sections \ref{sec:abstract-series} through \ref{sec:progressions}) is comprised of studying the abstract ghost series. In order to prove the results we are after, we assume a number of reasonable axioms that will be clarified in Section \ref{sec:abstract-series}. In Section \ref{sec:full-space}, we verify that these axioms are satisfied in the setting of \cite{BergdallPollack-GhostPaperShort} and in Section \ref{sec:rhobar-space}, we treat the case of a fixed $\rhobar$. Finally, in Section \ref{sec:conj} we state the $\rhobar$-ghost conjecture and give the computational evidence we have compiled thus far.

\subsection*{Notations}
Let $p$ be a prime number and write $q=p$ if $p$ is odd and $q = 4$ if $p=2$.  Set $\delta=p-1$ if $p$ is odd and $\delta = 2$ if $p=2$, so $\delta$ is the size of the torsion subgroup of $\Z_p^\times$. We will also use $v_p(-)$ to denote a $p$-adic valuation (on the $p$-adic complex numbers $\C_p$, say) normalized so $v_p(p) = 1$. If $P(t) = \sum_{i\geq 0} a_i t^i$ its a power series over $\C_p$, its Newton polygon $\NP(P)$ is the lower convex hull of the set of points $(i,v_p(a_i))$ in the standard $xy$-plane.

\subsection*{Acknowledgements}
We thank Frank Calegari and Toby Gee for pointing out that the method of Ash and Stevens \cite{AshStevens-Duke} can be used to explicitly write down dimension formulas for spaces of modular forms with a fixed residual representation. We also thank Liang Xiao for many helpful discussions related to the ghost conjecture.  We further thank the anonymous referees for several very helpful comments. Both authors finally thank the Max Planck Institute for Mathematics in Bonn for providing a stimulating environment during the period when much of this paper was written as well as for our use of their computer servers on which the computations of this paper were carried out.  The first author was supported by NFS grant DMS-1402005 and the second author was supported by NFS grants DMS-1303302 and DMS-1702178 as well as a fellowship from the Simons Foundation.

\section{Abstract ghost series}\label{sec:abstract-series}
The data needed for an abstract ghost series is two functions $\d,\dn: \Z \to \Z$ and an integer $k_0$ with $0 \leq k_0 < \delta$.  For notation, if $n$ is an integer then set $k_n = k_0 + n\delta$ and also define $\ddp = 2\d + \dn$. By the end of this section, we will make three assumptions \eqref{growth}, \eqref{LG}, and \eqref{ND} on $\d$ and $\dn$. (In Section \ref{sec:progressions} we introduce a fourth axiom \eqref{QL} which implies \eqref{growth} and \eqref{LG}.) We first assume:
\begin{equation}
\label{growth}
\tag{G}
\lim_{n \goto \infty} \d(n) = \infty \quad \text{and} \quad
\lim_{n \goto -\infty} \d(n) +\dn(n) = -\infty.
\end{equation}
The label \eqref{growth} indicates the word `growth.' 

There are two standard examples we have in mind. We will repeatedly specialize to these examples throughout, so we have visually separated out the commentary as follows.  Further discussion, and the proofs of any unproven assertions, can be found in Sections \ref{sec:full-space} and \ref{sec:rhobar-space}.

\begin{sex-full}
Let $N \geq 1$ be an integer and $p$ a prime such that $p \ndvd N$. If $k \geq 4$ and even, we set $D_k = \dim S_{k}(\Gamma_0(N))$ and $D_k^{\new} = \dim S_{k}(\Gamma_0(Np))^{p-\new}$. As is well-known, there are explicit combinatorial formulas for $D_k$ and $D_k^{\new}$ in terms of $k$ (with constants depending on $N$ and $p$). Fix $k_0$ even. Then, define $d(n) = D_{k_n}$ and $\dn(n) = D_{k_n}^{\new}$ when $k_n \geq 4$. If $k_n < 4$ and even, we define $d(n)$ and $\dn(n)$ by evaluating the combinatorial formulas at $k = k_n$. (These are thus not exactly dimensions of spaces of cuspforms).  It is clear that \eqref{growth} is satisfied. The function $d_p$ models the dimensions of forms of level $\Gamma_0(Np)$. \underline{Note}:\ We will be forced to assume that $pN > 3$ eventually, only excluding $(p,N) = (2,1)$ and $(3,1)$.
\end{sex-full}

\begin{sex-rhobar}
Assume that $N$ and $p$ are as in the previous example. Further assume that $\rhobar$ is a semi-simple and continuous representation $\rhobar: \Gal(\bar \Q/\Q) \rightarrow \GL_2(\Fpbar)$ such that the $\rhobar$-component $S_k(\Gamma_1(N))_{\rhobar}$ of $S_k(\Gamma_1(N))$ is non-zero. We will define $\d(n) = \dim S_{k_n}(\Gamma_1(N))_{\rhobar}$ and $\dn(n) = \dim S_{k_n}(\Gamma_1(N)\cap \Gamma_0(p))_{\rhobar}^{p-\new}$ when $k_n \geq 2$ and even. In Section \ref{subsec:rhobar-spaces}, we explain how to naturally extend $\d$ and $\dn$ to functions on all integers. \underline{Note}:\ We will always assume that $p \geq 5$ in this example, and we will also omit $\rhobar$ that are cyclotomic twists of $1 \oplus \omega$, $\omega$ being the mod $p$ cyclotomic character.
\end{sex-rhobar}

We now seek to define the abstract ghost series
\begin{equation}\label{eqn:ghost-eqn}
G_{\d,\dn,k_0}(w,t) = G(w,t) = \sum_{i=0}^\infty g_i(w)t^i
\end{equation}
attached to the data $\d$, $\dn$, and $k_0$. The variable $w$ is a coordinate on the $k_0$-th component $\mathscr W_{k_0}$ of the $p$-adic weight space. Recall this means that $\mathscr W_{k_0}$ is the rigid analytic space of continuous $p$-adic valued characters on $\Zp^\times$ whose action on the torsion in $\Zp^\times$ is raising to the $k_0$-th power; each integer $k_n$ defines a point $k_n \in \mathscr W_{k_0}$ given by the character $z \mapsto z^{k_n}$. 

The space $\mathscr W_{k_0}$ is seen to be an open unit disc by fixing a topological generator $\gamma$ for $1 + 2p \Zp$ and then providing $\mathscr W_{k_0}$ with the coordinate $w_\kappa = \kappa(\gamma)-1 \in \{v_p(w) > 0\}$, except in the case $p=2$ and $k_0 = 1$. In that case, we define $w_\kappa = \kappa(\gamma)\gamma^{-1} - 1$. In every case, $w_\kappa$ depends on $\gamma$ only up to isometry. If $p$ is odd, then $v_p(w_{k}) \geq 1$ for all integer weights $k\in\mathscr W_{k_0}$; the purpose of the normalization when $p=2$ is that $v_2(w_{k}) \geq 3$ for all $k$, also.

\begin{remark}\label{remark:w-to-k-change}
We note that $v_p(w_k-w_k') = 1 + v_p(2) + v_p(k-k')$ when $k,k' \in \mathscr W_{k_0}$.
\end{remark}

Now we (uniquely) define the coefficients $g_i(w)$ in \eqref{eqn:ghost-eqn}  by three  conditions. First, they are functions on $\mathscr W_{k_0}$ given by monic polynomials in $w$. Second, the zeroes of $g_i(w)$ are exactly those $w_{k_n}$ such that 
$$
\d(n) < i < \d(n) + \dn(n).
$$
Condition \eqref{growth} guarantees that each $g_i(w)$ has only finitely many zeroes. 

Finally, the multiplicity $m_i(k_n)$ of $w_{k_n}$ as a zero of $g_i(w)$ is given by
\begin{equation}\label{eqn:mult-pattern}
m_i(k_n) = \begin{cases} 
\ds i-\d(n) & \ds \text{for~} \d(n) < i \leq \d(n) + \frac{\dn(n)}{2};\\
\ds \d(n) + \dn(n)-i & \ds \text{for~}  \d(n) + \frac{\dn(n)}{2} \leq i < \d(n) + \dn(n).
\end{cases}
\end{equation}
That is, at the first and last index $i$ such that $g_i(w_{k_n}) = 0$, $w_{k_n}$ is a zero of order 1; at the second and second to last index, it vanishes to order 2; and so on.

\begin{definition}\label{defn:ghost-series}
The ghost series associated with $\d$, $\dn$, and $k_0$ is the series $G(w,t) = \sum_{i=0}^\infty g_i(w)t^i \in \Zp[w][[t]]$.
\end{definition}

If $\kappa$ is a $p$-adic weight lying in $\mathscr W_{k_0}$ then we will write $G_\kappa(t) = G(w_\kappa,t) \in \C_p[[t]]$. Being a $p$-adic power series in one variable, $G_\kappa$ has a Newton polygon that we denote by $\NP(G_\kappa)$.

\begin{sex-full}
The ghost series just defined is trivially almost the same as the $p$-adic ghost series of tame level $N$ defined in \cite[Definition 2.1]{BergdallPollack-GhostPaperShort}. A difference only occurs if $k_0=2$ and it has no effect on Newton polygons. Let us be more precise. 

Consider $k_0 = 2$, $\d(n) = \dim S_{k_n}(\Gamma_0(N))$ for $k_n \geq 4$ and then extend to $k_n < 4$ by explicit formulas (similarly for $\dn$). Then, $\dim S_2(\Gamma_0(N)) = d(0) + 1$ and $\dim S_2(\Gamma_0(Np))^{p-\new} = \dn(0) - 1$. When $p\neq 2$, it is straightforward to check the second equality in:
\begin{align*}
d(1) &= \dim S_{2 + (p-1)}(\Gamma_0(N))\\
 &= \dim S_2(\Gamma_0(N)) + \dim S_{2}(\Gamma_0(Np))^{p-\new}\\
 &= \d(0) + \dn(0).
\end{align*}
This shows that the ghost coefficient $g_{d(1)}$ defined either as above or as in \cite{BergdallPollack-GhostPaperShort} is the trivial function $1$. So, the `ghost zero' $w_{k_n} = w_2$ plays no role in calculating the slopes of the Newton polygons in either this article or our previous one. (We do not address $p=2$; the definition in \cite[Section 5]{BergdallPollack-GhostPaperShort} is more complicated.)
\end{sex-full}

The rest of this section is comprised of an analysis of the coefficients $g_i$ and our further assumptions on $\d$ and $\dn$.  We now make our second assumption:
\begin{equation}
\label{ND}
\tag{ND}
 \d, ~\d+\dn, \text{~and~} \ddp \text{~are~non-decreasing~functions}.
\end{equation}
The label \eqref{ND} refers to `non-decreasing'. To check \eqref{ND}, it is sufficient that $d$ and $d^{\new}$ are non-decreasing, but the condition as given is easier to check in practice.

\begin{sex-full}
\eqref{ND} is true when $pN > 3$. For instance, $\d$ and $\ddp$ are non-decreasing as long as there exists a non-zero modular form of weight $\delta$ and level $\Gamma_0(N)$.
\end{sex-full}

\begin{sex-rhobar}
We will verify \eqref{ND} as long as $p \geq 5$.
\end{sex-rhobar}

By \eqref{growth}, the set of integers $n$ such that $d(n) < i$ is non-empty and bounded above. So, we define
$$
\HZ(g_i) = \sup\{ n \in \Z \st \d(n) < i \}.
$$
Similarly, 
$$
\LZ(g_i) = \inf \{ n \in \Z \st  i < \d(n) + \dn(n) \}.
$$
Note that both $\HZ(g_i)$ and $\LZ(g_i)$ are well-defined by \eqref{growth}.
The notation $\HZ$ and $\LZ$ refers to `highest zero' and `lowest zero'. The notation is justified by the next lemma (which requires our assumption \eqref{ND}).

\begin{lemma}
\label{lemma:gi}
If $i\geq 0$, then 
$$
\{ n \in \Z \st g_i(w_{k_n})=0 \} =  \{\LZ(g_i), \LZ(g_i) +1, \dots, \HZ(g_i)\}.
$$
\end{lemma}

\begin{proof}
Under \eqref{ND}, it is clear that $\{n \in \Z \st  i < d(n) + d(n)^{\new}\} = \{\LZ(g_i), \LZ(g_i)+1,\dotsc\}$. Similarly, $\{n \in \Z \st d(n) < i\} = \{\dotsc, \HZ(g_i)-1, \HZ(g_i)\}$. The lemma follows immediately.
\end{proof}

Now we define $\ds \Delta_i = \frac{g_i}{g_{i-1}}$ and write $\ds \Delta_i = \frac{\Delta_i^+}{\Delta_i^-}$ where $\Delta_i^{\pm} \in \Z_p[w]$ are monic and co-prime.

\begin{lemma}
\label{lemma:del}
\leavevmode
\begin{enumerate}
\item The zeroes of $\Delta_i^\pm$ are simple.
\item $w_{k_n}$ is a zero of $\Delta_i^+$ if and only if $\d(n) < i \leq \d(n) + {1\over 2} \dn(n)$.
\item $w_{k_n}$ is a zero of $\Delta_i^-$ if and only if $\d(n) + {1\over 2}\dn(n) +1 \leq  i  \leq  \d(n) + \dn(n)$.
\end{enumerate}
\end{lemma}

\begin{proof}
The definition \eqref{eqn:mult-pattern} of the multiplicity pattern implies that the multiplicity of $w_{k_n}$ as a root of $g_{i}$ is one more, one less, or equal to the multiplicity of $w_{k_n}$ as a root of $g_{i-1}$. This proves part (a). The remaining two parts follow easily.
\end{proof}

The notation $\HZ$ and $\LZ$ naturally generalizes as follows:
\begin{align*}
\HZ(\Delta_i^+) &= \sup \{ n \in \Z : \d(n) < i\};\\  \LZ(\Delta_i^+) &= \inf \{ n \in \Z : i \leq  \d(n) + \frac{\dn(n)}{2}  \};\\
\HZ(\Delta_i^-) &= \sup \{ n \in \Z : \d(n) + \frac{\dn(n)}{2} + 1  \leq  i  \};\\
\LZ(\Delta_i^-) &= \inf \{ n \in \Z : i < \d(n) + \dn(n) \}.\nonumber
\end{align*}
These quantities are well-defined by \eqref{growth} and \eqref{ND}. If $g(w) \in \Zp[w]$ is monic and  its roots are in $v_p(w) > 0$, write $\lambda(g) = \deg(g)$. We extend $\lambda$ to the quotient of monic polynomials in the natural way.
\begin{lemma}
\label{lemma:describe_delta}
If $i \geq 0$, then
$$
\{ n \in \Z \st  \Delta_i^{\pm}(w_{k_n})=0 \} =  \{\LZ(\Delta_i^{\pm}), \LZ(\Delta_i^{\pm}) +1, \dotsc, \HZ(\Delta_i^{\pm})\},
$$
and $\lambda(\Delta_i^\pm) = \HZ(\Delta_i^\pm) - \LZ(\Delta_i^\pm) + 1$.
\end{lemma}

\begin{proof}
The first claim follows as in Lemma \ref{lemma:gi} (note that \eqref{ND} implies that $\frac{1}{2} \ddp = \d + {1\over 2} \dn$ is non-decreasing).  The second claim then follows from  Lemma \ref{lemma:del}(a).
\end{proof}

To ensure that the slopes on $\NP(G_\kappa)$ appear with finite multiplicity for each $\kappa$, we now impose our third condition on $d$ and $d^{\new}$:
\begin{equation}
\label{LG}
\tag{LG}
\d(n) \sim  \Cd n \quad\text{and}\quad \dn(n) \sim \Cdn n \quad\quad\quad\quad (A,B > 0).
\end{equation}
Here and below, if $F(n)$ and $G(n)$ are two functions defined on integers $n \gg 0$ we use possibly non-standard notation and say $F(n) \sim G(n)$ if $F(n) = G(n) + O(1)$.
The label \eqref{LG} refers to `linear growth'.

\begin{sex-full}
\eqref{LG} is satisfied with $A = {\delta \over 12}\mu_0(N)$ and $B = {\delta \over 12}(p-1)\mu_0(N)$ where $\mu_0(N)$ is the index of $\Gamma_0(N)$ inside $\SL_2(\Z)$. Note that $B = (p-1)A$.
\end{sex-full}

\begin{sex-rhobar}
\eqref{LG} holds with values $A,B > 0$ such that $B = (p-1)A$, so the relationship between $A$ and $B$ is the same as in the other example.
\end{sex-rhobar}
The asymptotic behaviors of the quantities above is as follows.
\begin{lemma}
\label{lemma:lambda_asymp}
We have:
\begin{enumerate}
\item $\ds\HZ(\Delta_i^+) \sim {1\over \Cd} \cdot i$;
\item $\ds \LZ(\Delta_i^+) \sim \HZ(\Delta_i^-) \sim {2 \over 2\Cd + \Cdn} \cdot i$;
\item $\ds \LZ(\Delta_i^-) \sim {1\over \Cd+\Cdn} \cdot i$;
\item $\ds \lambda(\Delta_i^+) \sim \frac{\Cdn}{\Cd(2\Cd+\Cdn)} \cdot i$;
\item $\ds \lambda(\Delta_i^-) \sim \frac{\Cdn}{(2\Cd+\Cdn)(\Cd+\Cdn)} \cdot i$;
\item $\ds \lambda(\Delta_i) \sim \frac{\Cdn^2}{\Cd(\Cd+\Cdn)(2\Cd+\Cdn)} \cdot i$.
\end{enumerate}
\end{lemma}

\begin{proof}
The first three parts are immediate from the definitions and (\ref{LG}).  The second three parts follow from the first three and Lemma \ref{lemma:describe_delta}.
\end{proof}

\begin{proposition}
\label{prop:entire}
For each $\kappa \in \mathscr W_{k_0}(\C_p)$, $G_\kappa(t) \in \C_p[[t]]$ is an entire series. In particular, each slope of $\NP(G_\kappa)$ appears with finite multiplicity.
\end{proposition}

\begin{proof}
It suffices to check that $\liminf_i \lambda(\Delta_i) = \infty$ (\cite[Lemma 2.4]{BergdallPollack-GhostPaperShort}) and that is immediate from part (f) of Lemma \ref{lemma:lambda_asymp}.
\end{proof}

\section{Distributions of slopes}\label{sec:dist}
\label{sec:distribution}

We assume the notations of the previous section. In particular, let $G$ be the ghost series attached to the data $\d$, $\dn$ and $k_0$ where $\d$ and $\dn$ satisfy \eqref{growth}, \eqref{ND}, and \eqref{LG}. Especially, we write $A$ and $B$ for the constants in \eqref{LG}. Recall that $k_n = k_0 + n\delta$ and $G_{k_n}(t) = G(w_{k_n},t)$.

\subsection{Statement of results}\label{subsec:distribution-results}
Write $s_1(k_n) \leq s_2(k_n) \leq \dotsb$ for the ordered list of slopes of $\NP(G_{k_n})$.  The following theorem gives an asymptotic description of the $i$-th slope $s_i(k_n)$. We will begin to write $q=p$ for $p$ odd and $q = 4$ if $p=2$.

\begin{theorem}
\label{thm:asymptotic}
\leavevmode
\begin{enumerate}
\item If $i \leq d(n)$, then
\begin{equation*}
s_i(k_n) = \frac{q}{p-1} \cdot \frac{\Cdn^2}{\Cd(\Cd+\Cdn)(2\Cd+\Cdn)} \cdot i + O(\log n).
\end{equation*}
\item If $d(n) < i \leq d(n) + d^{\new}(n)$, then
\begin{equation*}
s_i(k_n) = \frac{p}{(p-1)^2} \cdot \frac{\Cdn^2}{2\Cd(\Cd+\Cdn)} \cdot k_n+ O(\log n).
\end{equation*}
\item  If $i > d(n) + d^{\new}(n)$, then for every $\varepsilon > 0$
\begin{equation*}
s_i(k_n) = {q\over p-1}\cdot {B^2 \over A(A+B)(2A+B)}\cdot  i + O(i^{1/2+\varepsilon}).
\end{equation*}
\end{enumerate}
\end{theorem}

Part (b) of Theorem \ref{thm:asymptotic} deals with the slopes over the range where $g_i(w_{k_n}) = 0$, so it provides an asymptotic for the slope of the very long line we are forcing to appear in each $\NP(G_{k_n})$. The only difference between parts (a) and (c) is the error terms. 

We can also state two corollaries giving asymptotic formulas for the highest `old slope' and highest `classical slope.'

\begin{corollary}
\label{cor:highest}
We have 
$$
\displaystyle s_{\d(n)}(k_n)= \frac{p}{(p-1)^2} \cdot \frac{\Cdn^2}{(\Cd + \Cdn)(2\Cd + \Cdn)} \cdot k_n + O(\log n).
$$
and
$$
\displaystyle s_{\ddp(n)}(k_n)= \frac{p}{(p-1)^2} \cdot \frac{\Cdn^2}{\Cd (\Cd + \Cdn)} \cdot k_n + O(n^{1/2+\ve}).
$$
\end{corollary}

\begin{proof}
Note that $\displaystyle \d(n) \sim \Cd k_n / \delta$ and $\displaystyle \ddp(n) \sim (2\Cd+\Cdn) k_n / \delta$.  The result then follows from Theorem \ref{thm:asymptotic} and the amusing identity $q/\delta= p/(p-1)$.
\end{proof}

\begin{sex-either}
In either example, $B = (p-1)A$. Part (b) of Theorem \ref{thm:asymptotic} reduces to $s_i(k_n) = {1\over 2} k_n + O(\log n)$ which is consistent with the $p$-adic slope of a $p$-new cuspform being ${k-2\over 2}$. The second part of Corollary \ref{cor:highest} says
\begin{equation*}
s_{d_p(n)}(k_n) = k_n + O(n^{1/2+\ve}),
\end{equation*}
which is consistent with the highest $U_p$-slope of a classical modular form in weight $k$ being bounded by $k-1$. The first asymptotic in Corollary \ref{cor:highest} says that
\begin{equation*}
s_{d(n)}(k_n) = {k_n\over p+1} + O(\log n).
\end{equation*} 
This is consistent with investigations of Gouv\^ea (\cite{Gouvea-WhereSlopesAre}). Buzzard has suggested that perhaps we even have $s_{d(n)}(k_n) \leq {k_n-1\over p+1}$ in $\Gamma_0(N)$-regular situations (\cite[Question 4.9]{Buzzard-SlopeQuestions}).
\end{sex-either}

We now wish to normalize the ghost slopes in order to study their distribution.  To this end, we divide by the (asymptotically) highest `classical slope' $s_{d_p(n)}(k_n)$.  That is, define the $i$-th normalized slope
$$
\tilde{s}_i(k_n) := s_i(k_n) \cdot \left(\frac{p}{(p-1)^2} \cdot \frac{\Cdn^2}{\Cd (\Cd + \Cdn)} \cdot k_n \right)^{-1},
$$
and then consider the set
$
\mathbf x_{k_n} = \left\{ \tilde{s}_i(k_n)   \st 1\leq i \leq \ddp(n)
\right\}  \subseteq [0,\infty).
$
Let $\mu^{(p)}_{k_n}$ be the probability measure on $[0,\infty)$ uniformly supported on $\mathbf x_{k_n}$. (See \cite[Sections 1.1--1.2]{Serre-Equidistribution} for the notion of weak convergence and its relationship to equidistribution.)

\begin{corollary}\label{corollary:gouvea-distribution}
As $n\goto \infty$, the measures $\mu^{(p)}_{k_n}$ weakly converge to a probability measure $\mu^{(p)}$ on $[0,1]$ which is supported on  
$$
\left[0,\frac{\Cd}{2\Cd+\Cdn}\right] \cup \left\{\frac{1}{2}\right\} \cup \left[\frac{\Cd+\Cdn}{2\Cd+\Cdn},1\right].$$
We have $\mu^{(p)}(\set{1\over 2}) = {\Cdn\over 2\Cd+\Cdn}$ and the remainder is uniformly distributed (for the Lebesgue measure).  
\end{corollary}

\begin{proof}
By Theorem \ref{thm:asymptotic},  for $\dn(n) < i \leq \d(n) + \dn(n)$, $\tilde s_i(k_n) = {1\over 2} + O(\log n /n)$. This explains $\set{\frac{1}{2}}$ getting mass
$
\frac{\dn(n)}{\ddp(n)} \sim \frac{\Cdn}{2\Cd+\Cdn}
$ in the limit.

Similarly, by Corollary \ref{cor:highest}, $\twid s_{d(n)}(k_n) = {A\over 2A+B} + O(\log n /n )$. So Theorem \ref{thm:asymptotic} further implies that the $\tilde{s}_{i}(k_n)$ for $1 \leq i \leq d(n)$ become uniformly distributed between $0$ and $\frac{\Cd}{2\Cd+\Cdn}$ as $n \goto \infty$.  The last case of $\d(n) + \dn(n) < i \leq \ddp(n)$ follows similarly.  
\end{proof}
\begin{sex-either}
In either example $B = (p-1)A$, so the distribution $\mu^{(p)}$ is supported on
\begin{equation*}
\left[0,\frac{1}{p+1}\right] \cup \left\{\frac{1}{2}\right\} \cup \left[\frac{p}{p+1},1\right]
\end{equation*}
with $\{{1\over 2}\}$ getting mass ${p-1\over p+1}$ and the remaining connected intervals having equidistributed mass ${1\over p+1}$. This is consistent with the references given above. We note that no one has made large scale calculations $\rhobar$-by-$\rhobar$.
\end{sex-either}

To prove Theorem \ref{thm:asymptotic}, rather than working directly with the slopes $s_i(k_n)$, we consider the $\Delta$-slopes of $G_{k_n}$ defined now.

\begin{definition}
Let $P(t) = 1 + \sum_{i \geq 0} a_i t^i \in \C_p[[t]]$.   If $a_{i-1}, a_i \neq 0$, we define the $i$-th $\Delta$-slope of $P$ to be the slope of the line segment connecting $(i-1,v_p(a_{i-1}))$ to $(i,v_p(a_{i}))$.  Explicitly, the $i$-th $\Delta$-slope equals 
$v_p(a_i) - v_p(a_{i-1})$.
\end{definition}

Our strategy to prove Theorem \ref{thm:asymptotic} is to  prove the analogous statement about the $\Delta$-slopes of $G_{k_n}$ and then deduce the theorem. There are various technical points in the argument, so let us  sketch  the argument first. We will use the notations $\Delta_i$, $\Delta_i^{\pm}$, etc. from Section \ref{sec:abstract-series}.

First, the $i$-th $\Delta$-slope of $G_{k_n}$ equals
$$
v_p(\Delta_i(w_{k_n})) = v_p(\Delta_i^+(w_{k_n})) - v_p(\Delta_i^-(w_{k_n})).
$$
By Lemma \ref{lemma:describe_delta}, the $\Delta_i^\pm$ are very simple to describe:
$$
\Delta_i^\pm(w_{k_n}) = (w_{k_n} - w_{k_{b}})(w_{k_n} - w_{k_{b-1}}) \cdot \cdots \cdot (w_{k_n} - w_{k_{a}})
$$
where $a = \LZ(\Delta_i^\pm)$ and $b=\HZ(\Delta_i^\pm)$. When this product is non-zero, we estimate its valuation as follows:\ each term contributes at least 1, every $p$-th term contributes at least 2, every $p^2$-th term contributes at least 3, and so on.  Thus, for $p>2$, we have a rough estimate
$$
v_p(\Delta_i^\pm(w_{k_n})) \approx \lambda(\Delta_i^\pm)  + \frac{\lambda(\Delta_i^\pm)}{p} + \frac{\lambda(\Delta_i^\pm)}{p^2} + \dots = \frac{p}{p-1} \lambda(\Delta_i^\pm),
$$
and so $v_p(\Delta_i(w_{k_n})) \approx \frac{p}{p-1} \lambda(\Delta_i)$ (Lemma \ref{lemma:general_val} below makes this heuristic precise.) Using the asymptotic we established in part (f) of Lemma \ref{lemma:lambda_asymp}, we deduce
$$
v_p(\Delta_i(w_{k_n})) \approx \frac{p}{p-1} \cdot \frac{\Cdn^2}{\Cd(\Cd+\Cdn)(2\Cd+\Cdn)} \cdot i,
$$
which is the $\Delta$-slope version of the first part of Theorem \ref{thm:asymptotic}.

Handling the slopes in the range $\d(n) < i < \d(n) + \dn(n)$ requires more work as $g_i(w_{k_n})=0$ for such $i$.  We 
show, instead, that when $n$ is large enough the Newton polygon in weight $k_n$ has a straight line from index $\d(n)$ to index $\d(n) + \dn(n)$.  To estimate the slope of this line (`the semi-stable line') we remove the zero or pole of $\Delta_i$ at $w_{k_n}$ and apply the above analysis to the resulting functions.

Once we have asymptotic control over all of the $\Delta$-slopes, it is relatively straightforward to gain asymptotic control over the actual slopes.

\subsection{An estimate on $\Delta$-slopes}
We begin with the following elementary estimate. We write $\log_p$ for the logarithm base $p$. (Note:\ $\log_p$ is {\em not} a $p$-adic logarithm.)

\begin{lemma}\label{lemma:valuations}
Suppose that $\lambda >0$ is an integer and $y \in \Z$ such that either $y>0$ or $y<1-\lambda$.  
\begin{equation*}
\frac{\lambda}{p-1} - \ceil{\log_p(\lambda)} - 1 \leq \sum_{i=0}^{\lambda-1} v_p(y+i) \leq \frac{\lambda-1}{p-1} + \max \{ \floor{\log_p(|y|)}, \floor{\log_p(|y+\lambda-1|)}\}.
\end{equation*}
\end{lemma}
\begin{proof}
For a lower bound, note that 
$$
\sum_{i=0}^{\lambda-1} v_p(y+i) = v_p(\lambda!) + v_p\left(\binom{y+\lambda-1}{\lambda}\right)
\geq v_p(\lambda!).
$$
For an upper bound, set $e := \max_{0\leq i \leq \lambda-1} v_p(y+i)$ and choose $a$ such that $0 \leq a \leq \lambda-1$ with $v_p(y+a) = e$.  Then $v_p(y+a+j) = v_p(j)$ for $-a \leq j \leq \lambda-1-a$, and thus
\begin{align*}
\sum_{i=0}^{\lambda-1} v_p(y + i) 
&= v_p(a!) + e +  v_p\left((\lambda-1-a)!\right) \\
&=  e+  v_p((\lambda-1)!) - v_p\left( \binom{\lambda-1}{a}\right)
\leq  e+  v_p((\lambda-1)!).
\end{align*}
The lemma then follows from the classical bounds on valuations of factorials 
$$
\frac{n}{p-1} - \ceil{\log_p(n)} -1 \leq \displaystyle v_p(n!) \leq \frac{n}{p-1},
$$
and the bound $e \leq \max \{ \floor{\log_p(|y|)}, \floor{\log_p(|y+\lambda-1|)}\}$.
\end{proof}

If $\lambda > 0$ is an integer and $b \in \Z$, we now define a polynomial
\begin{equation*}
P_{b,\lambda}(w) = (w-w_{k_b})(w-w_{k_{b-1}})\dotsb (w-w_{k_{b-\lambda+1}}).
\end{equation*}
modeling $\Delta_i^{\pm}$.

\begin{lemma}\label{lemma:general_val}
For $n \in \Z$ such that $P_{b,\lambda}(w_{k_n}) \neq 0$, we have
\begin{equation*}
v_p(P_{b,\lambda}(w_{k_n})) = {q \lambda \over p-1} + O(\log \lambda,\log|n-b|).
\end{equation*}
\end{lemma}
\begin{proof}
We assume $p>2$ (the case of $p=2$ is similar). Note that $k_n-k_j = (n-j)(p-1)$ and so Remark \ref{remark:w-to-k-change} and the oddness of $p$ implies
\begin{equation*}
v_p(w_{k_n}-w_{k_j}) = 1 + v_p(k_n-k_j) = 1 + v_p(n-j).
\end{equation*}
In turn, this gives the equalities
\begin{equation*}
v_p(P_{b,\lambda}(w_{k_n})) = \sum_{j=b-\lambda+1}^b v_p(w_{k_n}-w_{k_j}) = \lambda + \sum_{j=b-\lambda+1}^b  v_p(n-j).
\end{equation*}
Finally, Lemma \ref{lemma:valuations} applies to $y = n-b-\lambda+1$ and the same $\lambda$ as here (since $P_{b,\lambda}(w_{k_n})\neq 0$) and the result follows.
\end{proof}

The poles and zeros of $\Delta_i$ are simple (Lemma \ref{lemma:del}) and if $\Delta_i(w_{k})$ is well-defined and non-zero then $v_p(\Delta_i(w_{k}))$ is the $i$-th $\Delta$-slope of $G_k$.  To give uniform estimates, we define
$$
\Delta_i^*(w_{k_n}) := \begin{cases}
(w-w_{k_n}) \Delta_i(w)|_{w=w_{k_n}}  & \text{if~}\Delta_i \text{~has~a~pole~at~}w_{k_n};\\
\displaystyle \frac{\Delta_i(w)}{w-w_{k_n}} \big|_{w=w_{k_n}}   & \text{if~}\Delta_i \text{~has~a~zero~at~}w_{k_n};\\
\Delta_i(w_{k_n}) & \text{otherwise}.
\end{cases}
$$

\begin{proposition}\label{prop:ith_newtonslope}
We have
\begin{equation*}
v_p\left(\Delta_i^{\ast}(w_{k_n})\right) = \frac{q}{p-1} \cdot \frac{\Cdn^2}{\Cd(\Cd+\Cdn)(2\Cd+\Cdn)} \cdot i + O( \log n, \log i).
\end{equation*}
\end{proposition}
\begin{proof}
It suffices to prove the result separately for pairs $(i,k_n)$ ranging over a finite number of disjoint domains. With this in mind, we focus only on the pairs $(i,k_n)$ such that $w_{k_n}$ is a zero of $\Delta_i^+$, leaving the rest for the reader. Lemma \ref{lemma:lambda_asymp} implies that $n=O(i)$ when $\Delta_i^+(w_{k_n}) = 0$. So, we need to show that
\begin{equation}\label{eqn:Deltaistar}
v_p(\Delta_i^{\ast}(w_{k_n})) = \frac{q}{p-1} \cdot \frac{\Cdn^2}{\Cd(\Cd+\Cdn)(2\Cd+\Cdn)} \cdot i + O(\log i).
\end{equation}
For notation, write $b_i^{\pm} = \HZ(\Delta_i^{\pm})$ and recall $\lambda^{\pm}_i = \lambda(\Delta_i^{\pm})$. By Lemma \ref{lemma:describe_delta},
\begin{equation*}
\Delta_i^+(w) = P_{b^+_i,\lambda'}(w)\cdot (w-w_{k_n}) \cdot P_{{n-1},\lambda''}(w)
\end{equation*}
where $\lambda' + \lambda'' = \lambda_i^+ - 1$. So, by the definition of $\Delta_i^{\ast}$ we have
\begin{equation*}
v_p\left(\Delta_i^{\ast}(w_{k_n})\right) = v_p\left(P_{b^+_i,\lambda'}(w_{k_n})\right) + v_p\left(P_{n-1,\lambda''}(w_{k_n})\right) - v_p\left(P_{b^-_i,\lambda_i^-}(w_{k_n})\right).
\end{equation*}
Since $\lambda_i^{\pm}$ and $b_i^{\pm}$ are both $O(i)$ (Lemma \ref{lemma:lambda_asymp}) and $n=O(i)$, so Lemma \ref{lemma:general_val} implies that
\begin{equation}\label{eqn:Deltaistar-2}
v_p(\Delta_i^{\ast}(w_{k_n})) = {q\over p-1}(\lambda' + \lambda'' - \lambda_i^-) + O(\log i) = {q\over p-1}(\lambda_i^+ - \lambda_i^-) + O(\log i).
\end{equation}
We finally use the asymptotic for  $\lambda_i^+ - \lambda_i^-$
\begin{equation}\label{eqn:Deltaistar-3}
\lambda_i^+ - \lambda_i^- \sim \frac{\Cdn^2}{\Cd(\Cd+\Cdn)(2\Cd+\Cdn)} \cdot i,
\end{equation}
 given by Lemma \ref{lemma:lambda_asymp}. Then, \eqref{eqn:Deltaistar} follows from combining \eqref{eqn:Deltaistar-2} and \eqref{eqn:Deltaistar-3}.
\end{proof}

\subsection{Handling the semi-stable line}

Since $g_i(w_{k_n})=0$ if $\d(n) < i < \d(n) + \dn(n)$, we know that the slopes $s_i(k_n)$ for $i$ in this range are equal.  Here, we show that for $n$ large enough, $\d(n)$ and $\dn(n)$ are indices of (consecutive) breakpoints on $\NP(G_{k_n})$ and that the slope of the connecting line is as claimed in Theorem \ref{thm:asymptotic}.

We begin with a lemma that separates out $\Delta$-slopes  in the `oldform' and `newform' ranges.

\begin{lemma}
\label{lemma:Gbounds}
For each $\ve > 0$, there exists an $n_\varepsilon$ such that for all $n \geq n_{\varepsilon}$ we have:
\begin{enumerate}
\item $\displaystyle v_p(\Delta_i(w_{k_n})) \displaystyle < \left(\frac{p}{(p-1)^2} \cdot \frac{\Cdn^2}{(\Cd+\Cdn)(2\Cd+\Cdn)} + \ve\right) k_n$ if $i\leq \d(n)$, and
\item $\displaystyle v_p(\Delta_{i}(w_{k_n})) \displaystyle > \left(\frac{p}{(p-1)^2} \cdot \frac{\Cdn^2}{\Cd(2\Cd+\Cdn)} - \ve\right)k_n$ for all $i \geq \d(n) + \dn(n)$.
\end{enumerate}
\end{lemma}

\begin{proof}
Both parts follow from Proposition \ref{prop:ith_newtonslope}.  For instance, if $i\leq d(n)$,  then $\Delta_i^{\ast}(w_{k_n}) = \Delta_i(w_{k_n})$. So, Proposition \ref{prop:ith_newtonslope} implies that $v_p(\Delta_{i}(w_{k_n}))$ grows no faster than
$$
\frac{q}{p-1} \cdot \frac{\Cdn^2}{\Cd(\Cd+\Cdn)(2\Cd+\Cdn)} \cdot \d(n) + O(\log n).
$$
As in the proof of Corollary \ref{cor:highest}, we have $d(n)q \sim {Ap\over (p-1)} k_n$ and the result follows.
\end{proof}

The next lemma describes the slope of the line connecting  $(i,v_p(g_i(w_{k_n})))$ for $i=\d(n)$ to the point with $i=\d(n)+\dn(n)$. To ease notation, we write $y_i(k_n)$ for $v_p(g_i(w_{k_n}))$.

\begin{lemma}
\label{lemma:ss_slope}
We have
\begin{equation}\label{eqn:ss_slope}
\frac{y_{\d(n)+\dn(n)}(k_n) - y_{d(n)}(k_n)}{\dn(n)} = \frac{p}{(p-1)^2} \cdot \frac{\Cdn^2}{2\Cd(\Cd+\Cdn)} \cdot k_n + O(\log n).
\end{equation}
\end{lemma} 

\begin{proof}
First, we apply Proposition \ref{prop:ith_newtonslope} to deduce
\begin{align}
y_{\d(n)}(k_n) \label{eqn:ydk-shorter}
&= \sum_{i=1}^{\d(n)} v_p(\Delta_i (w_{k_n})) 
= \sum_{i=1}^{\d(n)} \frac{q}{p-1} \cdot \frac{\Cdn^2}{\Cd(\Cd+\Cdn)(2\Cd+\Cdn)} \cdot i + O(\log n) \\
&= \frac{q}{p-1} \cdot \frac{\Cdn^2}{2\Cd(\Cd+\Cdn)(2\Cd+\Cdn)} \cdot \d(n)^2 + O(n\log n).\nonumber
\end{align}
Among the $i$ with $\d(n) < i < \d(n) + \dn(n)$, $w_{k_n}$ is a zero of $\Delta_i$ exactly as many times as it is a pole (by construction), and so
$$
\left( \prod_{i=\d(n)+1}^{\d(n)+\dn(n)}  \Delta_i\right)(w_{k_n}) = \prod_{i=\d(n)+1}^{\d(n)+\dn(n)} \Delta_i^*(w_{k_n}).
$$
Arguing as we did for \eqref{eqn:ydk-shorter}, Proposition \ref{prop:ith_newtonslope} gives
\begin{equation}\label{eqn:ydk-longer}
y_{\d(n)+\dn(n)}(k_n) = 
\frac{q}{p-1} \cdot \frac{\Cdn^2}{2\Cd(\Cd+\Cdn)(2\Cd+\Cdn)} \cdot (\d(n)+\dn(n))^2 + O(n \log n).
\end{equation}
Write $\mathrm{SS}$ for the left-hand side of \eqref{eqn:ss_slope}. Then, if we combine \eqref{eqn:ydk-shorter} and \eqref{eqn:ydk-longer}, and then divide by $\dn(n) = O(n)$, we see
\begin{align*}
\mathrm{SS} &= {q\over p-1}\cdot {B^2 \over 2A(A+B)(2A+B)}\cdot (2\d(n)+\dn(n)) + O(\log n)\\
&= {p\over (p-1)^2}\cdot {B^2\over 2A(A+B)}\cdot k_n + O( \log n).
\end{align*}
This proves the lemma.
\end{proof}

\begin{proposition}
\label{prop:break}
For $n\gg0$, $i=\d(n)$ and $i=\d(n)+\dn(n)$ are indices of (consecutive) breakpoints on $\NP(G_{k_n})$ and the slope of the line connecting these breakpoints is
\begin{equation*}
{p\over(p-1)^2}\cdot {B^2\over 2A(A+B)}\cdot k_n + O(\log n).
\end{equation*}
\end{proposition}
\begin{proof}
Since $A,B > 0$ we have that
$$
\frac{\Cdn^2}{(\Cd+\Cdn)(2\Cd+\Cdn)}
< 
\frac{\Cdn^2}{2\Cd(\Cd+\Cdn)} 
< 
\frac{\Cdn^2}{\Cd(2\Cd+\Cdn)}.
$$
So, the result follows from Lemmas \ref{lemma:Gbounds} and \ref{lemma:ss_slope} combined with the following formal lemma about Newton polygons (whose proof is left to the reader).
\end{proof}

\begin{lemma}
\label{lemma:three_part}
Consider a collection $\cP = \{(i,y_i) : i\geq 0\}$ such that $y_i \in \R_{\geq 0} \union \set{\infty}$ and $y_i = \infty$ if and only if $N_1<i<N_2$ for some $N_1, N_2 \geq 0$.  If $i < j$, set $\Delta_{i,j} = {y_j-y_i\over j-i}$, and set $\Delta_i := \Delta_{i-1,i}$.  Assume that there are constants $\gamma_i$ such that:
\begin{enumerate}
\item 
If $i \leq N_1$ then $\Delta_i \leq \gamma_1$; 
\item 
If $N_2 < i$ then $\Delta_i \geq \gamma_2$; and 
\item
$\gamma_1 < \Delta_{N_1,N_2} < \gamma_2$.
\end{enumerate}
Then, $N_1$ and $N_2$ are (consecutive) indices of break points of $NP(\cP)$.
\end{lemma}

Now we can complete the proof of Theorem \ref{thm:asymptotic}.

\subsection{Proof of Theorem \ref{thm:asymptotic}}
Part (b) of the theorem follows from Proposition \ref{prop:break}, which also allows us to assume in the sequel that $n$ is chosen large enough that $i=d(n)$ and $i=d(n)+d^{\new}(n)$ are breakpoints on the Newton polygon $\NP(G_{k_n})$.

It remains to handle cases (a) and (c) of the theorem. We first claim that consecutive breakpoints in the range covered by these parts of the theorem are not too far apart.  More precisely, for $i \leq \d(n)$ or $i \geq \d(n) + \dn(n)$, let $N:= N(i,n)$ denote the smallest index $j  \geq i$ of a breakpoint of the Newton polygon of $G_{k_n}$, and let $M:= M(i,n)$ denote the largest such index $j$ with $j<i$. Note that $M < N$ and either both are less than $d(n)$ or both are larger than $d(n) + d^{\new}(n)$.

\begin{claim}
\leavevmode
\begin{enumerate}[(I)]
\item $N(i,n) - M(i,n) = O(\log n)$ if $i < d(n)$.
\item For each $\varepsilon > 0$, $N(i,n) - M(i,n) = O(i^{1/2 + \varepsilon})$ if $i > d(n) + d^{\new}(n)$.
\end{enumerate}
In particular, $N(i,n)$ and $M(i,n)$ are $i + O(\log n)$ or $i + O(i^{1/2+\varepsilon})$ depending on the range of $i$'s. (This follows because $i$ lies between $M(i,n)$ and $N(i,n)$.)
\end{claim}

Assuming this claim for the moment, the proof of the theorem follows quickly.  Indeed, 
 the line connecting $(M,y_M)$ to $(N,y_N)$ has slope $s_i(k_n)$ and by construction of the Newton polygon,
\begin{equation}
\label{eqn:ineq}
v_p(\Delta_N(w_{k_n})) \leq  s_i(k_n) \leq v_p(\Delta_{M+1}(w_{k_n})).
\end{equation}
Next, by Proposition \ref{prop:ith_newtonslope}, we know that 
$$
v_p(\Delta_{M+1}(w_{k_n})) = \DD \cdot M(i,n) + O(\log M(i,n),\log n)
$$
\and
$$
v_p(\Delta_{N}(w_{k_n})) = \DD \cdot N(i,n) + O(\log N(i,n),\log n)
$$
with $\DD = \frac{q}{p-1} \cdot \frac{\Cdn^2}{\Cd(\Cd+\Cdn)(2\Cd+\Cdn)}$. Finally, our (as of yet unproven) claim implies that the asymptotics become (the same)
\begin{equation*}
Di + \begin{cases}
O(\log n) & \text{if $i \leq d(n)$;}\\
O(i^{1/2+\varepsilon}) & \text{if $i > d(n) + d^{\new}(n)$}.
\end{cases}
\end{equation*}
The theorem now follows from \eqref{eqn:ineq}.

Returning to the unproven claim, we first handle the case where $i \leq d(n)$.  Then  
we know that there is some constant $C$ such that for $i,n$ large enough
\begin{equation}\label{eqn:M-abs-bound}
| v_p(\Delta_{M+1}(w_{k_n})) - \DD \cdot M(i,n)| \leq C \max\{\log M(i,n),\log n\}.
\end{equation}
As $M(i,n)$ and $i$ are less than $d(n)\sim An$, we may replace $C$ and assume the bound on the right-hand side of \eqref{eqn:M-abs-bound} is $C\log n$. Similarly, we also have (for $C$ large enough)
\begin{equation*}
| v_p(\Delta_{N}(w_{k_n})) - \DD \cdot N(i,n)|  \leq C \log n.
\end{equation*}
By \eqref{eqn:ineq}, we have $v_p(\Delta_{M+1}(w_{k_n})) \geq v_p(\Delta_N(w_{k_n}))$ and so straightforward algebra (add and subtract terms) implies that
\begin{equation*}
D\cdot N(i,n) -D\cdot M(i,n) \leq 2C \log n.
\end{equation*}
This proves part (I) of the claim.

Now we handle claim (II), where $i > d(n) + d^{\new}(n)$. We write $y_i(k_n) = v_p(g_i(w_{k_n}))$. Since $n = O(i)$ in this case, we can use the same logic as in the proof of Lemma \ref{lemma:ss_slope}  and the asymptotic in Proposition \ref{prop:ith_newtonslope} to see
\begin{align*}
y_i(k_n)
&= 
\sum_{j=1}^{i} Dj + O(\log i)
= \frac{D^2}{2} \cdot i^2 + O(\log i!) 
= \frac{D^2}{2} \cdot i^2 + O(i\log i)
\end{align*}
where the last equality uses Stirling's approximation.
In particular, all of the Newton points of $(i,y_i(k_n))$ lie between the  graphs of the two functions $y = \frac{D^2}{2} x^2 \pm C x \log x$ for some $C>0$. Lemma \ref{lemma:region} below then implies that
$$
N(i,n) - M(i,n) = O(i^{1/2+\ve}).
$$
(Apply the lemma to $N=N(i,n)$ and $M=M(i,n)$ and use finally that $M \leq i$.)
This completes the proof of claim (II), and thus the proof overall.

\begin{lemma}
\label{lemma:region}
Let $a,b>0$ and consider the region $\mathcal R$ comprised of the points on or between the graphs of $y=ax^2 \pm b x \log x$.  Let $M<N$ and assume that $y_\star$ are chosen so that the line segment connecting $(M,y_M)$ and $(N,y_N)$ is contained in $\mathcal R$.  Then for each $\ve >0$, there is a constant $C>0$ depending only on $a$ and $b$ such that
$$
N - M \leq C M^{1/2+\ve}.
$$
\end{lemma}

\begin{proof}
We briefly explain the proof of this (complicated) calculus exercise. Write $u(x) = ax^2 + bx\log x$ and $\ell(x) = ax^2 - bx\log(x)$. Given $M > 1$, let $\mathscr L_M$ be the unique line segment with endpoint $(M,\ell(M))$, which lies tangent to the graph of $y = u(x)$ and whose other endpoint lies on the graph of $y=\ell(x)$.

\begin{figure}[htbp]
\begin{tikzpicture}[scale=.5]
\draw[domain=1.737:8.949,smooth,variable=\x,black,thick,dotted] plot ({\x},{.8099*(\x-3)+1.229});
\node[fill=white] at (7,4.8) {$\footnotesize  \mathscr L_M$};
\draw[domain=0.8:8,smooth,variable=\x,black,thick] plot ({\x},{.1*\x^2 + .1*\x*ln(\x)});
\draw[domain=.8:10,smooth,variable=\x,black,thick] plot ({\x},{.1*\x^2 - .1*\x*ln(\x)});
\draw[fill=black] (3,1.229) circle[radius=2.5pt];
\draw[fill=black] (1.737,.206) circle[radius=2.5pt];
\draw[fill=black] (8.949,6.047) circle[radius=2.5pt];
\node at (2.1,1.8) {$\footnotesize  (t,u(t))$};
\node at (2.5,-.4) {$\footnotesize  (M,\ell(M))$};
\node at (10.6,5.8) {$\footnotesize (N,\ell(N))$};
\node at (3.3,6.4) {$ \footnotesize  u(x) = ax + bx\log x$};
\node at (8.3,1.4) {$\footnotesize \ell(x) = ax - bx\log x$};
\end{tikzpicture}
\caption{}\label{fig:two-graphs}
\end{figure}

The horizontal length of a line segment is (by definition) the difference of the $x$-coordinates of its endpoints. We note that $\mathscr L_M$ is special because it has the largest horizontal distance among line segments completely contained in $\mathcal R$ and passing through a point of the form $(M,y_M)$, and so it is enough to assume that $(N,y_N) = (N,\ell(N))$. Write $(t,u(t))$ for the point on the graph $y=u(x)$ which is tangent to $\mathscr L_M$. (See Figure \ref{fig:two-graphs}.) With this notation, it is enough to confirm that
\begin{enumerate}
\item $t = M + O(M^{1/2+\varepsilon})$ for any $\varepsilon > 0$, and
\item $N = t + O(t^{1/2+\varepsilon})$ for any $\varepsilon > 0$.
\end{enumerate}
To show (a) we may show the stronger assertion that $t < M + M^{1/2+\varepsilon}$ when $M \gg 0$. Equivalently, by the choice of $t$, it is enough to show
\begin{equation}\label{eqn:slope-vs-secant}
u'(M + M^{1/2+\varepsilon}) > {u(M+M^{1/2+ \varepsilon})-\ell(M) \over M^{1/2+\varepsilon}}
\end{equation}
for $M \gg 0$. (Here, as usual, $u'$ is the derivative of $u$.) But as $M \goto \infty$, we have
\begin{equation*}
u'(M+M^{1/2+\varepsilon}) \approx 2a(M+M^{1/2+\varepsilon})
\end{equation*}
whereas
\begin{equation*}
{u(M+M^{1/2+ \varepsilon})-\ell(M) \over M^{1/2+\varepsilon}} \approx 2aM + aM^{1/2+\varepsilon}.
\end{equation*}
(Here, $f(M) \approx g(M)$ means that $f(M)/g(M) \goto 1$ as $M \goto \infty$.) This proves \eqref{eqn:slope-vs-secant} for $M \gg 0$ as claimed. Likewise, to show (b) it is enough to prove that
\begin{equation*}
u'(t) < {\ell(t+t^{1/2+\varepsilon})-u(t) \over t^{1/2+\varepsilon}}
\end{equation*}
when $t \gg 0$. This is proven in a completely analogous manner.
\end{proof}

\section{Arithmetic progressions}\label{sec:progressions}

\subsection{Statement of results}

In this section, under a new hypothesis on the functions $\d$ and $\dn$ (see \eqref{QL} below), we will show that the slopes of $\NP(G_\kappa)$ form a union of arithmetic progression, up to finitely many exceptions, for any $\kappa$ with $w_\kappa \nin \Zp$. 

\begin{definition}\label{defn:quasi-linear}
A function $d : \Z \goto \Z$ is quasi-linear if there exists positive integers $P_d, Q_d > 0$ such that $d(n+P_d) = d(n) + Q_d$ for all $n$. We call $P_d$ the period and $Q_d$ the defect (of $d$).
\end{definition}

We note the following obvious facts we will use throughout. First, if $d$ is quasi-linear then then $d(n) \sim (Q_d/P_d)n$. Second, the sum of two quasi-linear functions is quasi-linear.
 Third, if $d$ and $d'$ are quasi-linear with the same period and defect and $d(n) = d'(n)$ for $n \gg 0$, then $d = d'$. Fourth, if $d(n)$ is a function defined for $n \gg 0$ and $d(n+P_d) = d(n)+Q_d$ over its domain, then it can be uniquely extended to a quasi-linear function.

Now we make our new assumption on the functions $\d$ and $\dn$:
\begin{equation}
\label{QL}
\tag{QL}
\text{$\d$ and $\dn$ are quasi-linear}.
\end{equation}
By the previous paragraph, \eqref{QL} implies \eqref{LG} and it implies that $d+d^{\new}$ and $d_p$ are quasi-linear.  We now fix periods $P_{\star}$ and defects $Q_{\star}$ for $\star$ each of $\d$, $\dn$, $\d+\dn$, and $\ddp$.

\begin{sex-full}
We verify \eqref{QL} with periods and defects given by
$$
\Pd = \Pdn =  \ds \frac{12}{\gcd(12,\delta)}, 
~\Qd =
 \frac{\delta\mu_0(N)}{\gcd(12,\delta)} \and\Qdn = (p-1) \Qd.
$$
(This gives the constants $A,B$ in the condition \eqref{LG} that we previously claimed.) We will also show that we can take 
$$
\Pddn = \begin{cases}
\ds \frac{12}{\gcd(12,\delta)} & \text{if~}p>3\\
2 & \text{if~}p=3\\
3 & \text{if~}p=2
\end{cases},
\and 
\Qddn = 
 \begin{cases} 
 p\Qd & \text{if~}p>3\\
 \mu_0(N)  &\text{if~}p=2,3,
 \end{cases}
 $$
as well as
$$
\Pdp = \begin{cases}
1 & \text{if~}p>3\\
p  & p=2,3
\end{cases}
\and Q_{\ddp} = 
\begin{cases}
\ds {(p-1)(p+1)\mu_0(N)\over 12} & \text{if~} p>3\\
(p-1) \mu_0(N) & p=2,3.
\end{cases}
$$
\end{sex-full}
\begin{sex-rhobar}
We will verify \eqref{QL} and that the periods of $\d$ and $\dn$ can be taken to be $\Pd = \Pdn = p+1$. The defects will satisfy $(p-1)\Qd = \Qdn$. (And, $\Qd$ is effectively computable). For $\ddp$ we will have period $\Pdp = 1$ and $\Qdp = \Qd$. (This gives the constants $A,B$ in the condition \eqref{LG} that we previously claimed.)
\end{sex-rhobar}

We need some notation for the main theorem of this section.  First, if $w_{\kappa} \nin \Z_p$, we set 
$$
\alpha_{\kappa} = \sup_{w' \in \Z_p} v_p(w_{\kappa}-w') \in (0,\infty).
$$
If $v_p(w_\kappa) \not \in \Z$, then $\alpha_{\kappa}$ is simply $v_p(w_\kappa)$.  But, for example, if $w_\kappa = p + p^{3/2}$, then $\alpha_\kappa = 3/2$. Second, let 
\begin{equation*}
v_0 = \begin{cases}
1 & \text{if $p$ is odd},\\
3 & \text{if $p = 2$.}
\end{cases}
\end{equation*}
So, the integers $k \in \mathscr W_{k_0}$ all lie in $v_p(w_k) \geq v_0$. Finally, define
\begin{equation*}
\Qf := \lcm(\Qd,\Qdp,\Qddn),
\end{equation*}
and for each integer $r \geq 0$, write
\begin{equation}\label{eqn:defn-Qr}
\Qfr = \begin{cases}
p^r \Qf & \text{if $p$ is odd,}\\
\Qf & \text{if $p=2$ and $r < v_0$,}\\
2^{r-2}\Qf & \text{if $p=2$ and $r \geq v_0$}.
\end{cases}
\end{equation}
Note that $Q_r = Q$ if $r = 0$, regardless of $p$.

\begin{theorem}\label{theorem:global-halo-progressions}
Assume $\kappa \in \mathscr W_{k_0}$ and $w_{\kappa} \nin \Z_p$. Set $r = \floor{\alpha_\kappa}$. Then,
the slopes of $\NP(G_{\kappa})$  form a finite union of $\Qfr$-many arithmetic progressions with common difference
\begin{equation*}
\Qf \cdot \left( \frac{\Pd}{\Qd} - \frac{4 \Pdp}{\Qdp} + \frac{\Pddn}{\Qddn}\right) \cdot \left(\alpha_{\kappa} + \sum_{v=v_0}^r (p-1)p^{r-v}\cdot v \right),
\end{equation*}
up to finitely many exceptional slopes contained within the first $\Qfr$-many slopes.
\end{theorem}

\begin{remark}\label{rmk:Qd-even}
If $\Qdp$ is chosen to be even then Theorem \ref{theorem:global-halo-progressions} is true with $Q = \lcm(\Qd,\Qdp/2,\Qddn)$ instead (see Remark \ref{remark:why-Qd-even-is-ok}). This happens to be the case in the $\Gamma_0(N)$-level example, for instance, and gives an optimal version of the result.
\end{remark}
\begin{remark}
The conclusion of Theorem \ref{theorem:global-halo-progressions} is also true if $w_\kappa \in \Zp$ but $v_p(w_\kappa) < v_0$ (Corollary \ref{cor:AP-theorem-halo-cases}). This is more general only if $p=2$.
\end{remark}

We now unravel the complicated constants in Theorem \ref{theorem:global-halo-progressions} for our examples.

\begin{sex-full}
Taking Remark \ref{rmk:Qd-even} into account, the number of progressions predicted in Theorem \ref{theorem:global-halo-progressions} when $r=0$ is 
\begin{equation*}
Q = 
\begin{cases}
\ds {p(p-1)(p+1)\mu_0(N)\over 24} & \text{if~}p>2,\\
\mu_0(N) & \text{if~}p=2,
\end{cases}
\end{equation*}
while the common difference of these progressions is
$$
\Qf \cdot \left( \frac{\Pd}{\Qd} - \frac{4 \Pdp}{\Qdp} + \frac{\Pddn}{\Qddn} \right)
= 
\begin{cases}
\ds {(p-1)^2 \over 2} & \text{if~}p>2,\\
1 & \text{if~}p=2.
\end{cases}
$$
Thus, Theorem \ref{theorem:global-halo-progressions} matches what was claimed in \cite[Theorem 4.2]{BergdallPollack-GhostPaperShort}. A small calculation shows that the number of progressions and common difference is completely consistent with what follows from \cite[Section 3]{BergdallPollack-FredholmSlopes}.
\end{sex-full}

\begin{sex-rhobar}
The quantity $Q$ is equal to $p\Qdp$. Thus the common difference given in Theorem \ref{theorem:global-halo-progressions} when $r=0$ is equal to
\begin{equation*}
p\Qd\left({p+1\over \Qd} - {4 \over \Qd} + {p+1 \over p\Qd}\right) = (p-1)^2.
\end{equation*}
This is twice the corresponding value in the $\Gamma_0(N)$-level example. Numerically though it appears that if we combine the $\rhobar$-slopes with the $\rhobar \otimes \omega^{(p-1)/2}$-slopes, the arithmetic progressions mesh nicely and cause this common difference to be cut in half.
\end{sex-rhobar}

\begin{remark}
\label{rmk:wadic}
As explained in \cite[Section 4]{BergdallPollack-GhostPaperShort}, the ghost series satisfies a halo property. Namely if $v_p(w_\kappa) < v_0$, then $v_p(w_{\kappa}-w_{k_n}) = v_p(w_\kappa)$ for all $n$ and so the scaling ${1\over v_p(w_\kappa)}\NP(G_\kappa)$ is independent of $\kappa$. More precisely, write $\bar G(t) \in \Fp[[w]][[t]]$ for the mod $p$ reduction of $G(w,t)$. Then we can compute $\NP(\bar G)$ with respect to the $w$-adic valuation on $\Fp[[w]]$, and it is clear that the common value of ${1\over v_p(w_\kappa)}\NP(G_\kappa)$ is equal to $\NP(\bar G)$. So, Theorem \ref{theorem:global-halo-progressions} implies that the slopes of $\NP(\overline{G})$ form a union of arithmetic progressions up to finitely many exceptional slopes. (Actually we will show this result first in the text below. See Corollary \ref{cor:AP-theorem-halo-cases}.)
\end{remark}

The remainder of this section is devoted to the proof of  Theorem \ref{theorem:global-halo-progressions}.  Our strategy (like in Section \ref{sec:distribution}) is to first verify a corresponding statement for $\Delta$-slopes. Then, we deduce the theorem by the following general lemma on Newton polygons.

\begin{lemma}
\label{lemma:newton_slopes}
Consider a collection $\cP = \{(i,y_i) : i\geq 0\}$ such that $y_i \in \R_{\geq 0}$.  If the $\Delta$-slopes of $\cP$ form a union of $C$ arithmetic progressions with the same common difference, then the same holds for the slopes of $\NP(\cP)$ up to finitely many exceptional slopes contained within the first $C$ slopes.
\end{lemma}

\begin{proof}
This follows immediately from observing that if $x\geq C$ is the index of a breakpoint of $\NP(\cP)$, then $x-C$ is also the index of a breakpoint of $\NP(\cP)$.
\end{proof}

\subsection{Changes in $\lambda$-invariants}
Now we begin to use the assumption \eqref{QL}.
\begin{lemma}
\label{lemma:hzlz}
We have
\begin{enumerate}
\item $\HZ(\Delta^+_{i+\Qd}) = \HZ(\Delta^+_i) + \Pd$;
\item $\LZ(\Delta^+_{i+\Qdp}) = \LZ(\Delta^+_i) + 2\Pdp$;
\item $\HZ(\Delta^-_{i+\Qdp}) = \HZ(\Delta^-_i) + 2\Pdp$;
\item $\LZ(\Delta^-_{i+\Qddn}) = \LZ(\Delta^-_i) + \Pddn$.
\end{enumerate}
\end{lemma}

\begin{proof}
For part (a), \eqref{QL} implies that $n \mapsto n+\Pd$ gives a bijection
$$
\{ n ~:~ \d(n) < i \} \to \{ n ~:~ \d(n) < i + \Qd \}.
$$
Thus the supremums of these two sets differ by $\Pd$, proving the claim. The rest of the parts are similar.
\end{proof}

Before the next proposition, recall that  $\lambda_i^\pm = \lambda(\Delta_i^\pm)$ and $\lambda_i=\lambda(\Delta_i) = \lambda^+_i-\lambda^-_i$.

\begin{proposition}
\label{prop:lambda_change}
Set $\Qplus = \lcm(\Qd,\Qdp)$, $\Qminus=\lcm(\Qdp,\Qddn)$ and $\Qf=\lcm(\Qplus,\Qminus)$. Then,
\begin{enumerate}
\item 
$\lambda^+_{i+\Qplus} = \lambda^+_i + \frac{\Qplus}{\Qd} \cdot  \Pd - 
\frac{\Qplus}{\Qdp} \cdot  2\Pdp$,
\item 
$\lambda^-_{i+\Qminus} = \lambda^-_i + 
\frac{\Qminus}{\Qdp} \cdot  2\Pdp - \frac{\Qminus}{\Qddn} \cdot  \Pddn$,
\item 
$\lambda_{i + \Qf} = \lambda_i + \Qf \cdot ( \frac{\Pd}{\Qd} - \frac{4 \Pdp}{\Qdp} + \frac{\Pddn}{\Qddn})$.
\end{enumerate}
\end{proposition}

\begin{proof}
Parts (a) and (b) follow immediately from Lemma \ref{lemma:hzlz} and Lemma  \ref{lemma:describe_delta}. Part (c) follows from (a) and (b).
\end{proof}

\begin{remark}\label{remark:why-Qd-even-is-ok}
Suppose that $\Qdp$ is even. Then parts (b) and (c) of Lemma \ref{lemma:hzlz} can be replaced by the stronger statements that $\LZ(\Delta^+_{i+\Qdp/2}) = \LZ(\Delta_i^+) + \Pdp$ and $\HZ(\Delta^-_{i+\Qdp/2}) = \HZ(\Delta_i^-) + \Pdp$. In turn, Proposition \ref{prop:lambda_change} remains true with $\Qplus = \lcm(\Qd, \Qdp/2)$ and $\Qminus = \lcm(\Qdp/2,\Qddn)$. These changes do not alter the proof of Theorem \ref{theorem:global-halo-progressions} below, so this justifies Remark \ref{rmk:Qd-even}.
\end{remark}

\begin{corollary}\label{cor:AP-theorem-halo-cases}
Theorem \ref{theorem:global-halo-progressions} is true if $v_p(w_\kappa) < v_0$ (even if $w_\kappa \in \Zp$).
\end{corollary}
\begin{proof}
By Remark \ref{rmk:wadic} it is enough to prove the analogous statement about $\NP(\bar G)$ (for $\bar G$ described in that remark). But Proposition \ref{prop:lambda_change}(c) implies that the $\Delta$-slopes of $\NP(\bar G)$, which are the $\lambda_i$, form $Q$-many arithmetic progressions of common difference $\Qf \cdot ( \frac{\Pd}{\Qd} - \frac{4 \Pdp}{\Qdp} + \frac{\Pddn}{\Qddn})$. So the corollary follows from Lemma \ref{lemma:newton_slopes}.
\end{proof}

\subsection{Proof of Theorem \ref{theorem:global-halo-progressions}}
To give the proof of Theorem \ref{theorem:global-halo-progressions} (beyond Corollary \ref{cor:AP-theorem-halo-cases}) we will use the proposition above and the next three lemmas.

\begin{lemma}\label{lemma:minimal-valuations}
Suppose that $x_0 \in \mathcal O_{\mathbf C_p} - \Zp$. Then, for any choice of $x \in \Zp$ such that $v_p(x_0 - x) = \sup_{y\in\Zp} v_p(x_0 - y)$ we have
\begin{equation*}
v_p(x_0 - x') = \min(v_p(x_0-x),v_p(x-x'))
\end{equation*}
for all $x' \in \Zp$.
\end{lemma}
\begin{proof}
Let $s=v_p(x_0-x)$, $t=v_p(x-x')$, and $u=v_p(x_0-x')$.  By the choice of $x$, $u \leq s$.  If $u<s$, then $t=u<s$ by the ultrametric (in)equality, so $u=\min(s,t)$.  If $u=s$, then the ultrametric inequality implies $t\geq u=s$, so $u=\min(s,t)$ again.
\end{proof}

\begin{lemma}\label{lemma:easy}
Suppose that $k_1,\dotsc,k_M$ is an ordered list of integers which form an arithmetic progression of length $M$ and difference co-prime to $p$. Then, if $p^e \dvd M$ and $D = M/p^e$, we have
\begin{enumerate}
\item $\sizeof\set{k_i \st v_p(k_i) \geq e} = D$, and
\item if $0 \leq v < e$ then $\sizeof \set{k_i \st v_p(k_i) = v} = D\varphi(p^{e-v}) = D(p-1)p^{e-v-1}$.
\end{enumerate}
\end{lemma}
\begin{proof}
This is left to the reader.
\end{proof}

Now fix $\kappa \in \mathscr W_{k_0}$ such that $w_\kappa \nin \Zp$. Then, we define $\alpha_\kappa = \sup_{w \in \Zp} v_p(w_\kappa -w)$.
\begin{lemma}\label{lemma:k-plus-weight}
If $v_p(w_\kappa) \geq v_0$, then there exists an integer $k^+ \in \mathscr W_{k_0}$ such that $v_p(w_\kappa - w_{k^+}) = \alpha_\kappa$.
\end{lemma}
\begin{proof}
This follows from the density of the $w_k$ in $p\Zp$ if $p$ is odd or $8\Z_2$ if $p=2$.
\end{proof}

And now we complete the proof of Theorem \ref{theorem:global-halo-progressions}.

\begin{proof}[Proof of Theorem \ref{theorem:global-halo-progressions}]
Recall that $r = \floor{\alpha_\kappa}$. By Corollary \ref{cor:AP-theorem-halo-cases} we may assume that $v_p(w_\kappa)\geq v_0$. In particular, $r\geq 1$ if $p$ is odd and $r \geq 3$ if $p = 2$. Define $e = r$ for $p$ odd and $e = r-2$ for $p=2$, and then $Q_r = p^e Q$ regardless of $p$ ($Q_r$ as in \eqref{eqn:defn-Qr}). By Lemma \ref{lemma:k-plus-weight} we also fix an integer $k^+ \in \mathscr W_{k_0}$ such that $\alpha_\kappa = v_p(w_\kappa - w_{k^+})$. Then, from Lemma \ref{lemma:minimal-valuations} we have that
\begin{equation}\label{eqn:reminder-min}
v_p(w_{\kappa}-w_{k_n}) = \min(\alpha_\kappa, v_p(w_{k_n}-w_{k^+}))
\end{equation}
for all integers $n$.

We are going to apply Lemma \ref{lemma:newton_slopes} so our goal is to compare the $(i+\Qfr)$-th $\Delta$-slope with the $i$-th $\Delta$-slope in weight $\kappa$. That is, we must compare $v_p(\Delta_{i+\Qfr}(w_{\kappa}))$ with $v_p(\Delta_{i}(w_{\kappa}))$. To ease notation, define $\DD^\pm$ and $D$ to be those constants in Proposition \ref{prop:lambda_change} which satisfy 
$$
\lambda^\pm_{i+Q} = \lambda^\pm_i + \DD^\pm
\and
\lambda_{i+Q} = \lambda_i + \DD
$$
(so $D = D^+ - D^-$).  We observe that $\lambda_{i+\Qfr}^{\pm} =  \lambda_i^{\pm} + p^e \DD^\pm$, and so $\Delta^\pm_{i+\Qfr}$ has $p^e \DD^\pm$ more zeroes than $\Delta_i^\pm$. Thus we can write $\Delta_{i+\Qfr}^\pm = a^\pm\cdot b^\pm$ with $a^\pm,b^\pm \in \Z_p[w]$ products of linear factors with $a^\pm$ vanishing at largest $\lambda_i^{\pm}$-many zeros of  $\Delta_{i+\Qfr}^\pm$ and $b^\pm$ vanishing at the remaining $p^eD^\pm$-many.

Now let $h$ be one of the polynomials $h \in \set{a^{\pm},b^{\pm},\Delta_i^{\pm}}$. The roots of $h$ are all of the form $w_{k_n}$ for a consecutive list of integers $n$ (by the construction of $a^{\pm}$ and $b^{\pm}$ in particular). It also follows from \eqref{eqn:reminder-min} that
\begin{multline}\label{eqn:h-equation-halo}
v_p(h(w_\kappa)) = \alpha_\kappa \cdot \sizeof \set{k_n \st h(w_{k_n}) = 0 \text{ and } v_p(w_{k_n}-w_{k^+}) \geq r+1}\\
+ \sum_{v=0}^r v\cdot \sizeof \set{k_n \st h(w_{k_n}) = 0 \text{ and } v_p(w_{k_n} -w_{k^+}) = v}.
\end{multline}
\begin{claim}
$v_p(a^{\pm}(w_\kappa)) = v_p(\Delta_i^{\pm}(w_\kappa))$
\end{claim}
To prove the claim, first note that by construction $a^{\pm}$ and $\Delta_i^{\pm}$ have the same number of zeroes and both sets of zeroes form an arithmetic progression with the same common difference. Second, Lemma \ref{lemma:hzlz} implies that $\HZ(a^{\pm}) = \HZ(\Delta_{i+Q_r}^{\pm}) \congruent \HZ(\Delta_i^{\pm}) \bmod p^e$. It follows (remember Remark \ref{remark:w-to-k-change}) that the ordered lists of $w_{k_n}$ for which $a^{\pm}(w_{k_n}) = 0$ and $\Delta_i^{\pm}(w_{k_n}) = 0$ are congruent to each other (as ordered lists) modulo $p^{r+1}$ (regardless of $p$). So the claim follows by direct examination of the right-hand side of \eqref{eqn:h-equation-halo}. 

The claim being shown, we have that
\begin{equation*}
v_p\left(\Delta_{i+Q_r}^{\pm}(w_\kappa)\over \Delta_i^{\pm}(w_\kappa)\right) = v_p(b^{\pm}(w_\kappa)).
\end{equation*}
For clarity, let us assume that $p>2$ now. Then, $b^{\pm}$ has $p^r D^{\pm}$-many zeros $w_{k_n}$ and the $k_n$ lie in an arithmetic progression of difference co-prime to $p$. In particular, the same is true for the $p^rD^{\pm}$-many $k_n - k^+$. Since $p$ is odd we deduce from \eqref{eqn:h-equation-halo} and Lemma \ref{lemma:easy} that
\begin{equation*}
v_p(b^{\pm}(w_\kappa)) = D^{\pm}\left(\alpha_\kappa + \sum_{v=0}^r v \cdot (p-1)p^{r-v} \right).
\end{equation*}
The $v=0$ part of this sum is clearly zero, and since $D = D^+ - D^-$, we see that
\begin{equation*}
v_p\left(\Delta_{i+Q_r}(w_\kappa)\over \Delta_i(w_\kappa)\right) = D\left(\alpha_\kappa + \sum_{v=1}^r v(p-1)p^{r-v}\right).
\end{equation*}
Applying Lemma \ref{lemma:newton_slopes} finishes the proof when $p > 2$.

If $p=2$ one must be slightly more careful in the previous paragraph. The sequence $k_n-k^+$ is an arithmetic progression of length $2^e D^{\pm} = 2^{r-2} D^{\pm}$ but the common difference is $p=2$. However, the elements in the sequence are all even, so one can apply Lemma \ref{lemma:easy} to the sequence ${k_n-k^+\over 2}$ instead. Using that $v_2(w_{k_n}-w_{k^+}) = 2 + v_2(k_n-k^+)$, it is straightforward to check that from \eqref{eqn:h-equation-halo} that
\begin{equation*}
v_2(b^{\pm}(w_\kappa)) = D^{\pm}\left(\alpha_\kappa + \sum_{v=3}^r v \cdot 2^{r-v}\right).
\end{equation*}
The final result follows in the same manner as before.
\end{proof}

\section{Example:\ cuspforms of level $\Gamma_0(N)$}\label{sec:full-space}
Let $\Gamma = \Gamma_0(N)$ and $\Gamma_0 = \Gamma_0(Np)$ for $p \nmid N$. In this section, we verify that `dimensions' of spaces of cuspforms of level $\Gamma$ and $\Gamma_0$ satisfy the axioms \eqref{ND} and \eqref{QL}, with explicit constants.

To begin recall (\cite[Section 6.1]{Stein-ModularForms}) that for $k > 2$ and even, we have
\begin{footnotesize}
\begin{multline}\label{eqn:dim_Gamma0(N)}
\dim S_k(\Gamma)
=
\frac{(k-1)\mu_0(N)}{12} 
+ \left( \bfloor{\frac{k}{4}} - \frac{k-1}{4} \right) \mu_{0,2}(N)\\
+ \left( \bfloor{\frac{k}{3}} - \frac{k-1}{3} \right) \mu_{0,3}(N)
- \frac{c_0(N)}{2},
\end{multline}
\end{footnotesize}
and
\begin{footnotesize}
\begin{multline}
\label{eqn:dim_new}
\dim S_{k}(\Gamma_0)^{p-\new} = {(k-1)(p-1)\over 12}\mu_0(N) + \left(\bfloor{{k\over 4}} - {k-1\over 4}\right)\left(-1 + \left({-4 \over p}\right)\right)\mu_{0,2}(N)\\ + \left(\bfloor{{k\over 3}} - {k-1\over 3}\right)\left(-1 + \left({-3 \over p}\right)\right)\mu_{0,3}(N).
\end{multline}
\end{footnotesize}

Here we use standard notations:\ $\mu_0(N)$ is the index of $\Gamma_0(N)$ in $\SL_2(\Z)$, $c_0(N)$ is the number of cusps on $X_0(N)$, $\mu_{0,2}(N)$ is the number of order two elliptic points on $X_0(N)$, $\mu_{0,3}(N)$ is the number of order three elliptic points on $X_0(N)$, and $\left({a\over b}\right)$ is the Kronecker symbol.

For $k \in \Z$, write $D_k$ for the right-hand side of \eqref{eqn:dim_Gamma0(N)} and $D_k^{\new}$ for the right-hand side of \eqref{eqn:dim_new}.  Fix an even $k_0$ such that $2 \leq k_0 < p+1$, set $k_n = k_0 + n(p-1)$ and finally set
$$
\d(n) = D_{k_n}
\and
\dn(n) = D_{k_n}^{\new}.
$$
We will verify now that \eqref{QL} and \eqref{ND} hold, provided that $pN > 3$.

\subsubsection{The condition \eqref{QL}}

\begin{proposition}
\label{prop:fullQL}
The functions $\d$ and $\dn$  are quasi-linear. Moreover,
\begin{equation*}
P_d = P_d^{\new} = {12\over \gcd(12,\delta)} \and Q_{d} = {\delta \mu_0(N)\over \gcd(12,\delta)} \and Q_{d^{\new}} = (p-1)Q_d
\end{equation*}
for any $p$. We have the following periods and defects for $d+d^{\new}$ and $d_p$:
\begin{equation*}
P_{d+d^{\new}} = \begin{cases}
P_d & \text{if $p > 3$;}\\
2 & \text{if $p=3$;}\\
3 & \text{if $p=2$,}
\end{cases}
\and
Q_{d + d^{\new}} = \begin{cases}
pQ_d & \text{if $p > 3$;}\\
\mu_0(N) & \text{if $p=2,3$,}
\end{cases}
\end{equation*}
whereas
\begin{equation*}
P_{d_p} = \begin{cases}
1 & \text{if $p > 3$;}\\
p & \text{if $p=2,3$,}
\end{cases}
\and
Q_{d_p} = \begin{cases}
{p(p-1)(p+1)\over 12}\mu_0(N) & \text{if $p > 3$;}\\
(p-1)\mu_0(N) & \text{if $p=2,3$.}
\end{cases}
\end{equation*}
\end{proposition}

\begin{proof}
From \eqref{eqn:dim_Gamma0(N)} and \eqref{eqn:dim_new} we see $D_{k+12} = D_k + \mu_0(N)$ and $D_{k+12}^{\new} = D_k^{\new} + (p-1)\mu_0(N)$. This implies $d$ and $d^{\new}$ are quasi-linear and explains the values of $\Pd$, $\Qd$, $\Pdn$, and $\Qdn$ (for any $p$). Since $\d$ and $\dn$ are quasi-linear, it is clear that $\d+\dn$ is quasi-linear as well. Moreover, we can take $\Pddn =\Pd =\Pdn$ and $\Qddn =\Qd + \Qdn = p \Qd$. For $p=2,3$ one can do better and the claimed values in those cases follow from direct examination of \eqref{eqn:dim_Gamma0(N)} and \eqref{eqn:dim_new}.

Finally we consider $d_p = 2d + d^{\new}$. The quasi-linearity of $d_p$ is clear, but we can optimize the periods and defects beyond just realizing $d_p$ through its definition as a sum. For this, note that for $n\geq 0$ (or $n\geq 1$ if $k_0=2$), the value of $d_p(n)$ is the same as the right-hand side of \eqref{eqn:dim_Gamma0(N)} at $k=k_n$ and replacing $N$ by $Np$ (because that is how dimensions of spaces of cuspforms work). Since $d_p(n)$ is quasi-linear, this holds for all $n$ actually.

To derive the periods and defects of $d_p$ is now straightforward. In fact, after replacing $N$ by $Np$ one checks, case-by-case, that the right-hand side of \eqref{eqn:dim_Gamma0(N)} is invariant under $k \mapsto k + p-1$ if $p > 3$ and $k \mapsto k + 2p$ if $p=2,3$; the values of $Q_{d_p}$ are easy to determine as well. For instance, if $p \congruent 5 \bmod 12$, then $\mu_{0,3}(Np) = 0$ whereas $\bfloor{\frac{k}{4}} - \frac{k-1}{4}$ depends only on $k \bmod p-1$.
\end{proof}

\subsubsection{The condition \eqref{ND}}

\begin{proposition}
\label{prop:fullND}
If $pN>3$, then $\d$ and $\ddp$ are non-decreasing.
\end{proposition}
\begin{proof}
If $p>3$, let $f$ denote the weight $p-1$ Eisenstein series of level 1.  If $p=2$ or $3$, then $N > 1$ by assumption, so let $f$ denote a weight 2 Eisenstein series of level $N$. In either case, multiplication by $f$ yields an injection
$$
S_k(\Gamma') \hookrightarrow S_{k+\delta}(\Gamma')
$$
for $\Gamma'$ either $\Gamma$ or $\Gamma_0$.
Thus, $\d(n+1) \geq \d(n)$ and $\ddp(n+1) \geq \ddp(n)$ for $k_n \geq 4$.  Since $\d$ and $\ddp$ are quasi-linear (Proposition \ref{prop:fullQL}), the inequalities hold for all $n$.
\end{proof}

The Eisenstein trick in Proposition \ref{prop:fullND} does not apply to check that $\d + \dn$ is non-decreasing because multiplication by an Eisenstein series has no reason to preserve spaces of newforms. But, we have a different argument.

\begin{proposition}\label{prop:fullND-2}
The function $\d+\dn$ is non-decreasing.
\end{proposition}

\begin{proof}
Since $\d+\dn = \ddp-\d$, it suffices to see that
\begin{equation}
\label{eqn:check}
\ddp(n+1)-\ddp(n) \geq \d(n+1) - \d(n).
\end{equation}
First suppose that $p > 3$ and let $Q_d$ and $Q_{d_p}$ be as in Proposition \ref{prop:fullQL}. Then, the left-hand side of \eqref{eqn:check} is $\ddp(n+1)-\ddp(n) = Q_{d_p}$ whereas it follows from Proposition \ref{prop:fullND} that $Q_d \geq \d(n+1)-\d(n)$. Thus we want to show
\begin{equation*}
{(p-1)(p+1)\mu_0(N)\over 12} = Q_{d_p} \overset{?}{\geq} Q_d = {(p-1)\mu_0(N)\over \gcd(12,p-1)}.
\end{equation*}
This is clearly true if $p>3$, so we are finished in this case.

For $p=2,3$, one argues even more explicitly. Specifically, by \eqref{eqn:dim_Gamma0(N)} and \eqref{eqn:dim_new}, we have
\begin{small}
\begin{align*}
(D_{k+2} + D^{\new}_{k+2}) - (D_{k} + D^{\new}_{k})
&=
\begin{cases}
\frac{1}{3} \mu_0(N) - \left(\floor{\frac{k+2}{3}} - \floor{\frac{k}{3}} - \frac{2}{3} \right) \mu_{0,3}(N) & \text{if~}p=2; \\
\frac{1}{2} \mu_0(N) - \left(\floor{\frac{k+2}{4}} - \floor{\frac{k}{4}} - \frac{1}{2} \right) \mu_{0,2}(N) & \text{if~}p=3, 
\end{cases} \\
&\geq
\begin{cases}
\frac{1}{3} \mu_0(N) - \frac{1}{3} \mu_{0,3}(N) & \text{if~}p=2;\\
\frac{1}{2} \mu_0(N) - \frac{1}{2} \mu_{0,2}(N) & \text{if~}p=3.
\end{cases}
\end{align*}
\end{small}
Then, one ends by noting that $\mu_0(N) \geq \mu_{0,2}(N), \mu_{0,3}(N)$.
\end{proof}

\section{Example:\ $\rhobar$-components}\label{sec:rhobar-space}
Write $G_{\Q}$ for the absolute Galois group $\Gal(\bar \Q/\Q)$. Fix a decomposition group $D$ at $p$, and let $I \subset D$ be its inertia subgroup. We will use $\rhobar: G_{\Q} \rightarrow \GL_2(\Fpbar)$ to denote a continuous, odd, and semi-simple representation.  We also write $\omega$ for the mod $p$ cyclotomic character. Throughout Section \ref{sec:rhobar-space} we also make the following assumption:
\begin{equation}
\label{noE2}
\tag{No E2}
\rhobar \not\simeq (1 \oplus \omega)\otimes \omega^j \;\;\;\;\;\; (\text{for any $j$}).
\end{equation}
That is, $\rhobar$ is not a cyclotomic twist of the Galois representation associated with the `$E_2$-eigensystem.'

Unlike the previous section, we let $\Gamma = \Gamma_1(N)$ and $\Gamma_0 = \Gamma_1(N)\cap \Gamma_0(p)$ for $p \nmid N$. The spaces $S_k(\Gamma)$ and $S_k(\Gamma_0)^{p-\new}$ have $\Z$-linear bases and we write $S_k(\Gamma,\Zp)$ and $S_k(\Gamma_0,\Zp)^{p-\new}$ for the scalar extension of the corresponding $\Z$-modules to $\Zp$. These are finite free $\Zp$-modules. Write $\mathbf T$ for the commutative $\mathbf Z_p$-algebra generated by formal symbols $T_\ell$ as $\ell$ runs over primes $\ell \ndvd Np$. The algebra $\mathbf T$ acts by Hecke operators on many spaces, such as $S_k(\Gamma,\Zp)$. Further, each $\rhobar$ defines a canonical maximal ideal $\mathfrak m_{\rhobar} \subset \mathbf T$. Thus we may define $S_k(\Gamma,\Zp)_{\rhobar} = S_k(\Gamma,\Zp)_{\mathfrak m_{\rhobar}}$ and $S_k(\Gamma)_{\rhobar} = S_k(\Gamma,\Zp)_{\rhobar}[1/p]$ (and similarly for new spaces).

Now assume that $\rhobar$ is modular of level $N$, choose $k(\rhobar)$ to be the  least integer $k$ where $S_k(\Gamma)_{\rhobar}$ is non-zero. Set $k_0\geq 2$ to be the least integer $k_0 \congruent k(\rhobar) \bmod p-1$ and finally set $k_n = k_0 + n(p-1)$. We define $d(n) = \dim S_{k_n}(\Gamma)_{\rhobar}$ and $d^{\new}(n) = \dim S_{k_n}(\Gamma_0)^{p-\new}_{\rhobar}$ for $n\geq 0$. 

We will show below that these functions are quasi-linear on their domains and thus we can extend $d$ and $\dn$ canonically to all $n$. After doing that, we will check the axiom \eqref{ND} is satisfied. In all cases we will make the periods of quasi-linearity explicit; when $\rhobar$ is irreducible but reducible upon restriction to $D$, we will do the same for the defects of quasi-linearity (see Section \ref{subsec:buzzard-regular}).

In order to carry out the analysis, we need to use modular symbols. This is the topic of the next subsection.

\subsection{Recollection of modular symbols}\label{subsec:mod-symbs}

Throughout this subsection we use $g$ to denote a non-negative integer. Let $R = \Zp$, $\Fp$, or $\Qp$ and consider $\Sym^g(R^2)$ as homogenous polynomials of degree $g$ in variables $X$ and $Y$, equipped with a right action of $\gamma \in \GL_2(R)$ 
\begin{equation*}
P|_\gamma(X,Y) = P(dX - cY, aY - bX).
\end{equation*}
Any subgroup of $\SL_2(\Z)$ naturally maps to $\GL_2(R)$, so $\Sym^g(R^2)$ is endowed with the structure of right $\Gamma'$-module with $\Gamma' = \Gamma$ or $\Gamma_0$.  Thus we can consider the finite $R$-modules given by the cohomology $H^i_c(\Gamma', \Sym^g(R^2))$ ($i=1,2$). We note two things:\ $H^1_c(\Gamma', \Sym^g(R^2))$ is torsion-free and either cohomology is equipped with an $R$-linear action of the $\Zp$-algebra $\mathbf T$, so one can localize at the ideal $\mathfrak m_{\rhobar}$ as above. 

We recall that $\Sym^g(\Fp^2)$ is completely reducible under the action of $\Gamma_0$. Namely, write $B \subset \GL_2(\Fp)$ for the Borel subgroup of upper-triangular matrices. Then, if we set $F_{-1} = (0)$ and for $0 \leq j \leq g$ we define $F_j = F_{j-1} + \Fp X^{g-j} Y^{j}$ then it is straightforward to check that $(0) \sci F_0 \sci F_1 \sci \dotsb \sci \Sym^g(\Fp^2)$ is a $B$-stable filtration by $\Fp$-vector spaces whose consecutive quotients are one-dimensional. The action of $B$ on the $j$-th quotient is 
\begin{equation*}
F_j / F_{j-1} \simeq \Fp(a^jd^{g-j})
\end{equation*}
where $\Fp(a^s d^r)$ means the $B$-representation  sending  $\begin{smallpmatrix} a & b \\ 0 & d \end{smallpmatrix} \in B$ to $a^s d^{r}$. We note that this implies that the subquotients $H^i_c(\Gamma_0,F_j/F_{j-1})$ of the cohomology $H^i_c(\Gamma_0,\Sym^g(\Fp^2))$ are also Hecke-stable.

\begin{lemma}\label{lemma:Gamma0eis}
$H^2_c(\Gamma_0,\Fp(a^sd^r)) = (0)$ unless $s \congruent r \bmod p-1$. If $s \congruent r \bmod p-1$ then
\begin{equation*}
H^2_c(\Gamma_0,\Fp(a^sd^r))  = H^2_c(\Gamma_0,\Fp((ad)^r))  \simeq \F_p,
\end{equation*}
with the action of $T_\ell$ through $\ell^{r} + \ell^{1+r}$.
\end{lemma}
\begin{proof}
Let $M = \F_p(a^sd^r)$. By Poincar\'e duality, $H^2_c(\Gamma_0,M) \simeq H_0(\Gamma_0,M)$ is the largest quotient of $\F_p(a^sd^r)$ on which $\Gamma_0$ acts trivially. This proves the vanishing claim and it proves that $H_0(\Gamma_0,\F_p((ad)^r)) = \F_p((ad)^r)$ as a Hecke module. If $\mathbf e$ is a basis of $\F_p((ad)^r)$ then we have
\begin{equation*}
T_\ell(\mathbf e) =  \mathbf e|_{\begin{smallpmatrix} \ell \\ & 1 \end{smallpmatrix}} + \sum_{b=0}^{\ell-1} \mathbf e|_{\begin{smallpmatrix}1 & b \\ & \ell\end{smallpmatrix}} = (\ell^r + \ell^{1+r})\mathbf e.
\end{equation*}
This completes the proof.
\end{proof}
Now recall we assume $\rhobar$ satisfies  \eqref{noE2}.

\begin{lemma}\label{lemma:eis}
If $\Gamma' = \Gamma$ or $\Gamma_0$ and $M = \Sym^g(\Z_p^2)$ or $\Sym^g(\F_p^2)$, then $H^2_c(\Gamma',M)_{\rhobar} = (0)$.
\end{lemma}
\begin{proof}
As in the previous lemma, it suffices to show that $H_0(\Gamma',M)_{\rhobar} = (0)$. Since there is a natural quotient map $H_0(\Gamma_0,M) \twoheadrightarrow H_0(\Gamma,M)$, $H_0$ is right exact and the $H_0$'s are both finite over $\Z_p$, it suffices to let $\Gamma' = \Gamma_0$ and $M = \Sym^g(\Fp^2)$.

In that specific case,  Lemma \ref{lemma:Gamma0eis} and the assumption \eqref{noE2} imply together that $H_0(\Gamma_0,F_j/F_{j-1})_{\rhobar} = (0)$ for all $j$. By the right-exactness of $H_0(\Gamma_0,-)$ and descending induction on $0\leq j \leq g$, we deduce $H_0(\Gamma_0,M/F_{j-1})_{\rhobar} = (0)$ as well. Taking $j = 0$ proves our claim.
\end{proof}

\begin{proposition}\label{prop:Zp-vs-Fp}
If $\Gamma' = \Gamma$ or $\Gamma_0$, then
\begin{equation*}
\rank H^1_c(\Gamma', \Sym^g(\Zp^2))_{\rhobar} = \dim H^1_c(\Gamma', \Sym^g(\Fp^2))_{\rhobar}.
\end{equation*}
\end{proposition}
\begin{proof}
In general there is a canonical exact sequence
\begin{equation*}
0 \rightarrow H^1_c(\Gamma',\Sym^g(\Zp^2)) \otimes_{\Zp} \Fp \rightarrow H^1_c(\Gamma', \Sym^g(\Fp^2)) \rightarrow H^2_c(\Gamma', \Sym^g(\Zp^2)).
\end{equation*}
The proposition then follows by localizing at $\rhobar$ and applying Lemma \ref{lemma:eis}. 
\end{proof}

We now recall a result of Ash and Stevens (\cite{AshStevens-Duke}) to study the ranks in Proposition \ref{prop:Zp-vs-Fp}. For $g \geq 0$ we let $I_g = \Ind_{\Gamma_0}^{\Gamma}(\F_p(a^g))$. This only depends on $g \bmod p-1$.

\begin{theorem}[Ash--Stevens]\label{theorem:ash-stevens}
\leavevmode
\begin{enumerate}
\item If $0 < g < p$ then there is a natural short exact sequence
\begin{equation*}
0 \goto H^1_c(\Gamma,\Sym^g(\F_p^2))_{\rhobar} \goto H^1_c(\Gamma,I_g)_{\rhobar} \goto H^1_c(\Sym^{p-1-g}(\F_p^2))_{\rhobar \otimes \omega^{-g}} \goto 0.
\end{equation*}
\item If $g > p$, then there is a natural short exact sequence
\begin{equation*}
0 \goto H^1_c(\Gamma,\Sym^{g-(p+1)}(\F_p^2))_{\rhobar\otimes \omega^{-1}} \goto H^1_c(\Gamma,\Sym^g(\F_p^2))_{\rhobar} \goto H^1_c(\Gamma,I_g)_{\rhobar} \goto 0.
\end{equation*}
\end{enumerate}
\end{theorem}
\begin{proof}
This theorem follows from \cite[Theorem 3.4(a,c)]{AshStevens-Duke}, except our normalizations are slightly different (both in the action and in that we use cohomology with compact supports). But as in \cite[Lemma 3.2]{AshStevens-Duke}, we can construct canonical short exact sequences of $\GL_2(\F_p)$-modules
\begin{equation*}
0 \goto \Sym^g(\F_p^2) \goto I_g \goto \Sym^{p-1-g}(\F_p^2)\otimes \F_p({\det}^g)\goto 0 \;\;\;\;\;\;\;\;\; (0 < g < p),
\end{equation*}
and
\begin{equation*}
0 \goto \Sym^{g-(p+1)}(\F_p^2) \otimes \F_p({\det}) \overset{\theta}{\rightarrow} \Sym^g(\F_p^2) \goto I_g \goto 0 \;\;\;\;\;\;\;\;\; (g >p).
\end{equation*}
The results (a) and (b) now follow from the long exact sequence in $H^i_c$ and the vanishing of the relevant $H^2_c$'s by Lemma \ref{lemma:eis}.
\end{proof}

For $t \in \Z/(p-1)\Z$ and $g \geq 0$ we now define
\begin{equation*}
D(g,\rhobar,t) = \dim H^1_c(\Gamma,\Sym^g(\Fp^2))_{\rhobar\otimes \omega^t}.
\end{equation*}
Theorem \ref{theorem:ash-stevens}(b) has the following immediate consequence.

\begin{corollary}\label{cor:quasi-linearity-modsymb}
Assume that $g \geq p^2-1$. Then, for each $t$,
\begin{equation*}
D(g,\rhobar,t) = D(g-(p^2-1),\rhobar,t) + \sum_{i=0}^{p-2} \dim H^1_c(\Gamma,I_{g-2i})_{\rhobar \otimes \omega^{t-i}}.
\end{equation*}
\end{corollary}
\begin{proof}
It follows by induction on $j$ and Theorem \ref{theorem:ash-stevens}(b) that if $1 \leq j \leq p-1$, then
\begin{equation*}
D(g,\rhobar,t) - D(g-j(p+1),\rhobar,t-j) = \sum_{i=0}^{j-1} \dim H^1_c(\Gamma,I_{g-2i})_{\rhobar \otimes \omega^{t-i}}.
\end{equation*}
Taking $j = p-1$ gives the claim.
\end{proof}

By Corollary \ref{cor:quasi-linearity-modsymb}, $g \mapsto D(g,\rhobar,t)$ is quasi-linear with period $p^2-1$ and an explicit defect (for fixed $\rhobar$, $t$). We now aim to show the same for
\begin{equation*}
D_p(g,\rhobar,t) = \dim H^1_c(\Gamma_0,\Sym^g(\Fp^2))_{\rhobar\otimes \omega^t}.
\end{equation*}
In fact, the periods of $D_p$ and $D$ will be different, but their defects will be the same.

\begin{proposition}\label{prop:Symg-decompose}
If $g \geq 0$, then
\begin{equation*}
\dim H^1_c(\Gamma_0, \Sym^g(\Fp^2))_{\rhobar} = \sum_{j=0}^{g} \dim H^1_c(\Gamma_0, \Fp(a^{2j-g}))_{\rhobar \otimes \omega^{j-g}}.
\end{equation*}
\end{proposition}
\begin{proof}
If $0 \leq j \leq g$, then there is a natural exact sequence
\begin{equation*}
0 \rightarrow H^1_c(\Gamma_0,F_j/F_{j-1})_{\rhobar} \rightarrow H^1_c(\Gamma_0,\Sym^g(\Fp^2)/F_{j-1})_{\rhobar} \rightarrow H^1_c(\Gamma_0,\Sym^g(\Fp^2)/F_j)_{\rhobar} \rightarrow 0
\end{equation*}
(exactness on the right follows from Lemma \ref{lemma:Gamma0eis}). We deduce by induction that
\begin{equation*}
\dim H^1_c(\Gamma_0,\Sym^g(\Fp^2))_{\rhobar} = \sum_{j=0}^g \dim H^1_c(\Gamma_0,F_j/F_{j-1})_{\rhobar}.
\end{equation*}
The proposition then follows from observing that 
\begin{equation*}
F_j/F_{j-1} = \Fp(a^j d^{g-j}) = \Fp(a^{2j-g}(ad)^{g-j}),
\end{equation*}
and so
\begin{equation*}
H^1_c(\Gamma_0,F_j/F_{j-1})_{\rhobar} \simeq H^1_c(\Gamma_0,\Fp(a^{2j-g}))_{\rhobar \otimes \omega^{j-g}}.
\end{equation*}
This completes the proof.
\end{proof}

\begin{corollary}\label{cor:Dp-quasi-linearity}
If $g \geq p-1$, then
\begin{equation*}
D_p(g,\rhobar,t) = D_p(g-(p-1),\rhobar,t) + \sum_{i=0}^{p-2} \dim H^1_c(\Gamma,I_{g-2i})_{\rhobar\otimes \omega^{t-i}}
\end{equation*}
\end{corollary}
\begin{proof}
We note first that Shapiro's lemma implies $H^1_c(\Gamma,I_g) \simeq H^1_c(\Gamma_0,\F_p(a^g))$ for any $g$. Thus Proposition \ref{prop:Symg-decompose} implies that
\begin{equation*}
D_p(g,\rhobar,t) -  D_p(g-(p-1),\rhobar,t)= \sum_{j=g-(p-1)+1}^{g} \dim H^1_c(\Gamma,I_{2j-g})_{\rhobar\otimes \omega^{t+j-g}}.
\end{equation*}
Now this sum differs from the sum we want just by a change of indexing.
\end{proof}

\subsection{$\rhobar$-component of cuspforms}
\label{subsec:rhobar-spaces}
We turn now towards studying dimensions of spaces of $\rhobar$-components of cuspforms. For $k\geq 2$, $\Gamma'=\Gamma$ or $\Gamma_0$, and $\rhobar$ fixed we write $\mathcal E_k(\Gamma')_{\rhobar}$ for the $\rhobar$-component of the space of Eisenstein series of weight $k$. By the classification of Eisenstein series,  and because $\rhobar$ satisfies \eqref{noE2}, the function $k \mapsto \dim \mathcal E_k(\Gamma')_{\rhobar}$ depends only on $k \bmod 2$, and $2\dim \mathcal E_k(\Gamma)_{\rhobar} = \dim \mathcal E_{k}(\Gamma_0)_{\rhobar}$. We note that these spaces vanish if $\rhobar$ is irreducible.
\begin{lemma}\label{lemma:ES}
 If $\Gamma' = \Gamma$ or $\Gamma_0$, and $k\geq 2$ is an integer, then
\begin{equation*}
\dim \mathcal E_k(\Gamma')_{\rhobar} + 2 \dim S_k(\Gamma')_{\rhobar} =  \dim H^1_c(\Gamma',\Sym^{k-2}(\Fp^2))_{\rhobar}.
\end{equation*}
\end{lemma}
\begin{proof}
First note that $H^1_c(\Gamma',\Sym^{k-2}(\Zp^2))_{\rhobar} \subset H^1_c(\Gamma',\Sym^{k-2}(\Qp^2)_{\rhobar}$ is an isomorphism after inverting $p$. By Eichler--Shimura (see \cite[Proposition 2.5]{BellaicheDasgupta}) we deduce that 
\begin{equation*}
\dim \mathcal E_k(\Gamma')_{\rhobar} + 2 \dim S_k(\Gamma')_{\rhobar} = \rank H^1_c(\Gamma',\Sym^{k-2}(\Zp^2))_{\rhobar}.
\end{equation*}
The result as stated now follows from Proposition \ref{prop:Zp-vs-Fp}.
\end{proof}

Now we turn to the data needed for an abstract ghost series. We will make a slight switch in notation and write $\rbar$ for a continuous, odd, and semi-simple Galois representation modulo $p$ that is modular of  level $N$. That is, $S_k(\Gamma)_{\rbar}$ is non-zero for some $k$.  The previous results will be applied to twists $\rhobar = \rbar \otimes \omega^t$ into our setup. 

In this direction, if $t \in \Z/(p-1)\Z$ and $k$ is an integer such that $k \congruent k(\rbar) + 2t \bmod p-1$, we define
\begin{align*}
S(k,t) &= \dim S_{k}(\Gamma)_{\rbar \otimes \omega^t};\\
S(k,t)^{\new} &= \dim S_k(\Gamma_0)^{p-\new}_{\rbar \otimes \omega^t}.
\end{align*}
(The condition on $k$ is necessary to get non-zero values.) Further, write $k_{0,t}$ for the least integer $2 \leq k < p+1$ such that $k_{0,t} \congruent k(\rbar) + 2t \bmod p-1$. Then, for $n\geq 0$ define $k_{n,t} = k_{0,t} + n(p-1)$. Finally, the data we will use for an abstract ghost series will be 
\begin{align*}
d_t(n) &= S(k_{n,t},t);\\
d_t^{\new}(n) &= S^{\new}(k_{n,t},t).
\end{align*}
We note that these functions are defined only on $n\geq 0$. In the next subsections, we show that they are quasi-linear, so that we can extend them fully to $n \in \Z$. Our method for that requires defining $d_{p,t}(n) = \dim S_{k_{n,t}}(\Gamma_0)_{\rbar \otimes \omega^t}$, which satisfies $d_{p,t} = 2d_t + d^{\new}_t$ on its domain. After that, we check the condition \eqref{ND} is satisfied as well.

\subsubsection{The condition \eqref{QL}}
\begin{proposition}\label{prop:quasi-linear-galoisreps}
\leavevmode
\begin{enumerate}
\item The function $d_t$ is quasi-linear on $n\geq 0$, with period $P_{d_t} = p+1$;
\item The function $d_{p,t}$ is quasi-linear on $n\geq 0$, with period $P_{d_{p,t}} = 1$;
\item The defects $Q_{d_t}$ and $Q_{d_{p,t}}$ of quasi-linearity for $d_t$ and $d_{p,t}$ are equal.
\end{enumerate}
\end{proposition}
\begin{proof}
We prove all three statements at the same time. Recall the notation $D$ and $D_p$ from the end of Section \ref{subsec:mod-symbs}.

By Lemma \ref{lemma:ES}, we have
\begin{equation*}
d_t(n) = {1\over 2}\left(D(k_n-2,\rbar,t) -\dim \mathcal E_{k_n}(\Gamma)_{\rbar\otimes \omega^t}\right),
\end{equation*}
and
\begin{equation*}
d_{p,t}(n) = {1\over 2}\left(D_p(k_n-2,\rbar,t) - \dim \mathcal E_{k_n}(\Gamma_0)_{\rbar\otimes \omega^t}\right).
\end{equation*}
The dimensions of the Eisenstein series only depend on $k_n \bmod 2$, so they are independent of $n$ (since $p$ is odd). Thus the claims (a) through (c) follow from Corollary \ref{cor:quasi-linearity-modsymb} and Corollary \ref{cor:Dp-quasi-linearity}. (The defects of quasi-linearity can be seen to be positive using Theorem \ref{theorem:ash-stevens}(a); see the proof of Corollary \ref{cor:qldt}.)
\end{proof}

\begin{remark}
The defect $Q_{d_t}$ is always effectively computable. In Section \ref{subsec:buzzard-regular} below, we will study $Q_{d_t}$ when $\rbar$ is globally irreducible, but reducible upon restriction to a decomposition group at $p$.
\end{remark}

\begin{corollary}
$d_t^{\new}$ is quasi-linear. In particular, so is $d_t + d_t^{\new}$. The periods and defects  may be taken to be
\begin{align*}
P_{d_t^{\new}} &= p+1; & P_{d_t+d_t^{\new}} &= p+1;\\
Q_{d_t^{\new}} &= (p-1)Q_{d_t}; & Q_{d_t+d_t^{\new}} &= pQ_{d_t}.
\end{align*}
\end{corollary}
\begin{proof}
We note that $d_t^{\new} = d_{p,t} - 2d_t$ is a sum of quasi-linear functions, hence quasi-linear. Moreover, it can be taken to have period $P_{d_t^{\new}} = P_{d_t} = p+1$, in which case its defect is $Q_{d_t^{\new}} = (p-1)Q_{d_t}$. By the same logic, $d_t+d_t^{\new}$ is quasi-linear with the claimed periods and defects.
\end{proof}

We reiterate that having now checked all the functions above are quasi-linear, we extend their defintions to all $n \in \Z$.

\subsubsection{The condition \eqref{ND}}
\begin{proposition}\label{prop:dt-ND}
$d_t$ is non-decreasing.
\end{proposition}
\begin{proof}
This is a more robust version of the proof of Proposition \ref{prop:fullND}. Specifically, if $k\geq 2$, then
$$
\dim S_{k}(\Gamma) = \dim_{\Fp} S_k(\Gamma,\Fp)
$$
where $S_k(\Gamma,\Fp)$ is the space of mod $p$ modular forms. Then, since $p>3$, multiplication by the Eisenstein series $E_{p-1}$ yields a Hecke-equivariant injection
$$
S_k(\Gamma,\Fp)_{\rhobar} \hookrightarrow S_{k+p-1}(\Gamma,\Fp)_{\rhobar}.
$$
Thus $\dim S_{k+p-1}(\Gamma)_{\rhobar}\geq \dim S_{k}(\Gamma)_{\rhobar}$ if $k\geq 2$ and hence $\dt(n+1)\geq \dt(n)$ for $n \geq 0$. Since $\dt(n)$ is quasi-linear in $n$, it follows that $\dt(n+1)\geq \dt(n)$ in general.
\end{proof}

\begin{proposition}
Both $\ddpt$ and $\dt + \dnt$ are non-decreasing.
\end{proposition}

\begin{proof}
By Proposition \ref{prop:quasi-linear-galoisreps}, we have $\ddpt(n+1) = \ddpt(n) + \Qdpt$. This implies that $\ddpt$ is non-decreasing. Further, $\dt + \dnt = \ddpt - \dt$, so to prove that $\dt + \dnt$ is non-decreasing we must verify that
$$
\ddpt(n+1) - \ddpt(n) \geq \dt(n+1) - \dt(n).
$$
The left hand side of this inequality equals $\Qdpt$.  Since $\dt$ is non-decreasing (Proposition \ref{prop:dt-ND}), the right hand side is at most $\Qdt$.  Since $\Qdt = \Qdpt$ (Proposition \ref{prop:quasi-linear-galoisreps} again) we are done.
\end{proof}

\subsection{The Buzzard regular case}\label{subsec:buzzard-regular}

Consider the following `Buzzard regular' condition on a $\rhobar$:
\begin{equation}
\label{BR}
\tag{BR}
\rhobar|_D \text{ is reducible}.
\end{equation}
Our goal here is to make the defects of quasi-linearity somewhat explicit under \eqref{BR}. Following the notation of the previous section, we write $\rhobar = \rbar \otimes \omega^t$, with $t \in \Z/(p-1)\Z$. We further impose the following condition:\ $\rbar$ is irreducible, and
\begin{equation}\label{eqn:rbar-special-form}
\rbar|_I \sim \begin{pmatrix} \omega^{k(\bar r)-1} & \ast \\ 0 & 1\end{pmatrix}.
\end{equation}
Since each twist $\rbar \otimes \omega^t$ is irreducible, the weight part of Serre's conjecture (\cite{Edixhoven-Weights}) implies we do not need to distinguish between the weights $k(\rbar\otimes \omega^t)$ and the explicit recipe for $k(\rbar \otimes \omega^t)$ given in \cite{Serre-SurLesRep}.

\begin{lemma}
\label{lemma:krbar_twists}
\leavevmode
\begin{enumerate}
\item
If $\rbar$ is non-split at $p$, then 
$
k(\rbar \otimes \omega^t) \leq p+1 \iff t = 0.
$
\item
If $\rbar$ is split at $p$, then 
$
k(\rbar \otimes \omega^t) \leq p+1 \iff t = 0 \tv{or} t=p-k(\rbar).
$
\end{enumerate}
\end{lemma}
\begin{proof}
If $\rhobar = \rbar \otimes \omega^t$ then Serre states in \cite[Section 2.7]{Serre-SurLesRep} that $k(\rhobar)\leq p+1$ if and only if $\rhobar$ has a one-dimensional quotient on which inertia acts trivially. The equivalences we gave are apparent then.
\end{proof}

Recall that $S(k,t) = \dim S_k(\Gamma)_{\rbar\otimes \omega^t}$. Lemma \ref{lemma:krbar_twists} and Serre's conjecture allows us to determine $S(k,t)$ for $k \leq p+1$.

\begin{proposition}
\label{prop:base}
Assume that $k \leq p+1$. Then, the following conclusions hold.
\begin{enumerate}
\item
If $\rbar$ is non-split at $p$ and $k(\rbar) \neq 2$, then
$$
S(k,t) \neq 0 \iff (k,t) = (k(\rbar),0).
$$
\item
If $\rbar$ is non-split at $p$ and $k(\rbar) = 2$, then
$$
S(k,t) \neq 0 \iff (k,t) = (2,0) \tv{or} (p+1,0).
$$
\item
If $\rbar$ is split at $p$ and $k(\rbar) \neq 2$, then
$$
S(k,t) \neq 0 \iff (k,t) = (k(\rbar),0) \tv{or} (p+1-k(\rbar),p-k(\rbar)).
$$
\item
If $\rbar$ is split at $p$ and $k(\rbar) = 2$, then 
$$
S(k,t) \neq 0 \iff (k,t) = (2,0) \tv{or} (p-1,p-2) \tv{or} (p+1,0).
$$
\end{enumerate}
\end{proposition}

\begin{proof}
Assume $S(k,t) \neq 0$ with $k \leq p+1$.  Since $S(k,t) \neq 0$ implies that $k \geq k(\rbar \otimes \omega^t)$,  Lemma \ref{lemma:krbar_twists} implies that $t=0$ in the non-split case and $t=0$ or $t=p-k(\rbar)$ in the split case.
Further, since $\det(\rbar\otimes \omega^t)|_I = \omega^a$ where $a=k(\rbar) - 1 +  2t \bmod p-1$, we must have that $k \equiv k(\rbar) + 2t \pmod{p-1}$. This explains the possible values of $(k,t)$ in each case.

We now catalog references for the converses. The fact that $S(k(\rbar),0) \neq 0$ is a tautology. The fact that $S(p+1,0) \neq 0$ when $k(\rbar) = 2$ follows from the proof of Proposition \ref{prop:dt-ND} (multiplication by $E_{p-1}$). Finally, the fact that $S(p+1-k(\rbar),p-k(\rbar))\neq 0$ in the case $\rbar$ is split follows from the existence of companion forms (\cite{Gross-TamenessCriterion}).
\end{proof}

We know that $d_t$ is quasi-linear with period $p+1$, but now we want to see its exact defect. To ease notation, set
\begin{align*}
m_1 &= 
\begin{cases}
S(k(\rbar),0) & \text{if~} k(\rbar) \neq p+1;\\
S(2,0) & \text{otherwise};
\end{cases}\\
m_2 &= S(p+1,0);\\
m_3 &= S(p+1-k(\rbar),p-k(\rbar)).
\end{align*}
\begin{remark}
When $k(\rbar) = p+1$ (which only happens in the case of a `tr\`es ramifi\'e' extension) we have $m_1 = S(2,0) = 0$. We prefer to write it as $S(2,0)$ for the proof of the next statement.
\end{remark}
\begin{corollary}
\label{cor:qldt}
$\dt$ is quasi-linear with $\Pdt = p+1$ and 
$$
\Qdt = 
\begin{cases}
2m_1 & \tv{if} \rbar \tv{is~non-split~and} k(\rbar) \not\equiv 2 \bmod{p-1};\\
m_1 + m_2 & \tv{if} \rbar \tv{is~non-split~and} k(\rbar)  \equiv 2 \bmod{p-1};\\
2m_1 + 2m_3 & \tv{if} \rbar \tv{is~split~and} k(\rbar) \not\equiv 2 \bmod{p-1};\\
m_1 + m_2 + 2m_3 & \tv{if} \rbar \tv{is~split~and} k(\rbar) \equiv 2 \bmod{p-1}.\\
\end{cases}
$$
\end{corollary}

\begin{proof}
If $g \in \Z$ then we define $\alpha(g)$ to be the unique integer satisfying $\alpha(g) \congruent g \bmod p-1$ and $2 < g \leq p+1$. Since $\rbar$ is irreducible, Lemma \ref{lemma:ES} and Theorem \ref{theorem:ash-stevens}(a) combine to show that
\begin{equation*}
{1\over 2}\dim H^1_c(\Gamma,I_g)_{\rbar\otimes \omega^t} = S(\alpha(g+2),t) + S(p+3-\alpha(g+2),t-g)
\end{equation*}
for any integer $g$ and any $t \in \Z/(p-1)\Z$. Returning to the expression for $Q_{d_t}$ implicit in the proof of  Proposition \ref{prop:quasi-linear-galoisreps}, and the statement of Corollary \ref{cor:quasi-linearity-modsymb}, we deduce an explicit formula
\begin{equation}
\label{eqn:final-recursion-sum} 
Q_{d_t} = \sum_{j=0}^{p-2} S(\alpha(k(\rbar)+2j),j) +S(p+3-\alpha(k(\bar r)+2j),-k(\bar r) - j +2).
\end{equation}
To compute $Q_{d_t}$ we can now proceed case-by-case using Proposition \ref{prop:base} because the values in \eqref{eqn:final-recursion-sum} involve weights at most $p+1$.  

When $\rbar$ is non-split with $k(\rbar) \not\equiv 2 \pmod{p-1}$, the first term being summed above only contributes when $j=0$ while the second term only contributes if $j \equiv 2-k(\rbar) \pmod{p-1}$.  In each case, the contribution is $m_1$ and thus we can take $\Qdt = 2m_1$.

Suppose that $\rbar$ is non-split with $k(\rbar) \equiv  2 \pmod{p-1}$. Note that $\alpha(k(\rbar)) = \alpha(2) = p+1$. So, the first term contributes $m_2 = S(p+1,0)$ while the second contributes $m_1 = S(2,0)$. Thus we can take $\Qdt = m_1 + m_2$.

The remaining two cases are handled similarly. We leave them for the reader. 
\end{proof}

\begin{remark}
It is also important to have a base case for the function $d_{p,t}$, i.e.\ we need to calculate $d_{p,t}(0)$. This can readily be done using Proposition \ref{prop:Symg-decompose} combined with the proof of Corollary \ref{cor:qldt}.
When $\rbar$ is non-split and $k(\rbar) \not\equiv 2 \pmod{p-1}$, one has that
$$
\dim S_{k_{0,t}}(\Gamma_0)_{\rbar \otimes \omega^t} = 
\begin{cases}
2m_1 & \text{if~} t \leq  k_{0,t}-2; \\
0 & \text{otherwise}.
\end{cases}
$$
For other $\rbar$, the possible values of $S_{k_{0,t}}$ appear to be $0$, $2m_1$, $2m_2$, or $2m_1+2m_2$.  But we do not see a tidy way to express this short of a bunch of inequalities.  (It is trivial to program the answer on a computer though.)
\end{remark}

\section{A ghost conjecture for Buzzard regular $\rhobar$'s}\label{sec:conj}

We now formulate a $\rhobar$-version of the ghost conjecture.

\subsection{Statement of the conjecture}

Let $p \geq 5$ denote a  prime and, as before, let $\rhobar : G_{\Q} \to \GL_2(\Fpbar)$ denote  a continuous, odd, and semi-simple Galois representation that satisfies the condition \eqref{noE2}.  

Let $k_0$ denote the unique integer such that $0 \leq k_0 < p-1$ and $\det(\rhobar) = \omega^{k_0-1}$.  Let $N$ denote an integer which is prime-to-$p$ and $\rhobar$ is modular of level $N$.  Define $\d$ and $\dn$ as in our running $\rhobar$-component examples where $\d(n) = \dim S_{k_n}(\Gamma_1(N))_{\rhobar}$ and $\dn(n) = \dim S_{k_n}(\Gamma_0)^{p-\new}_{\rhobar}$ (here $\Gamma_0 = \Gamma_1(N)\cap \Gamma_0(p)$ as before).  Let $G_{\rhobar} = G_{\rhobar,N}$ denote the ghost series attached to $k_0$, $\d$, and $\dn$ which we will refer to as the $\rhobar$-ghost series.  For $\kappa$ a weight in $\mathscr W_{k_0}(\C_p)$, let $G_{\rhobar}(\kappa)$ denote the specialization of $G_{\rhobar}$ to weight $\kappa$.

Let $S^{\dag}_{\kappa}(\Gamma_0)$ denote the space of overconvergent cuspforms of weight $\kappa$ and tame level  $\Gamma_1(N)$, and let $S^{\dag}_{\kappa}(\Gamma_0)_{\rhobar}$ be the $\rhobar$-isotypic subspace.  Let $P_{\rhobar}(\kappa)$ denote the characteristic power series of $U_p$ on $S^{\dag}_{\kappa}(\Gamma_0)_{\rhobar}$. Before stating the conjecture, we recall that the `Buzzard regular' condition \eqref{BR} is a synonym for $\rhobar$ being locally reducible at $p$.
\begin{conjecture}
\label{conj:rhobarghost}
If $\rhobar$ satisfies \eqref{BR}, then $\NP(P_{\rhobar}(\kappa)) = \NP(G_{\rhobar}(\kappa))$ 
for all $\kappa \in \mathscr W_{k_0}(\C_p)$.
\end{conjecture}

Before discussing evidence, let us address the condition \eqref{noE2}. Under \eqref{noE2}, Proposition \ref{prop:quasi-linear-galoisreps} implies that the ghost series can be explicitly constructed after only a finite computation (which becomes shorter after Section \ref{subsec:buzzard-regular}). This allows us to make the numerical tests that follow.

On the other hand, we currently cannot efficiently determine the slopes of the abstract ghost series for $\rhobar=1\oplus \omega$ (or its cyclotomic twists). Out of  prudence, we have omitted making a conjecture in this case, although we know of no reason for Conjecture \ref{conj:rhobarghost} to fail in this case.

Further, we note that even if we remove the \eqref{noE2} hypothesis, it is not a priori clear that Conjecture \ref{conj:rhobarghost} for all locally reducible $\rhobar$ implies the ghost conjecture of \cite{BergdallPollack-GhostPaperShort}.   Indeed, let $G_i$ for $i=1,2$ denote abstract ghost series defined by functions $\d_i$ and $\dn_i$.  It is not generally true that the abstract ghost series defined by $\d_1+\d_2$ and $\dn_1+\dn_2$ has slopes equal to the union of the slopes of $G_1$ and $G_2$.  Nonetheless, we believe that such a statement will hold for $\rhobar$-ghost series defined on the same component of weight space (and could be proven with a little combinatorial care).  
With such a statement in hand, then the full ghost conjecture does indeed follow from the $\rhobar$-ghost conjecture.

\subsection{Numerical evidence}
Along with the theoretical evidence following from the theorems in Sections \ref{sec:distribution} and \ref{sec:progressions}, we also made extensive numerical tests of Conjecture \ref{conj:rhobarghost}.
The analysis in Section \ref{subsec:buzzard-regular} suggests that we consider five different `behaviors' for a fixed  residual representation $\rbar$ with small Serre weight as in \eqref{eqn:rbar-special-form}:
\begin{enumerate}[(A)]
\item \label{item:generic} $\rbar$ is non-split at $p$ and $2< k(\rbar)<p+1$;
\item \label{item:b} $\rbar$ is non-split at $p$ and $k(\rbar)=2$;
\item \label{item:c} $\rbar$ is non-split at $p$ and $k(\rbar)=p+1$;
\item \label{item:d} $\rbar$ is split at $p$ and $2< k(\rbar)<p+1$;
\item $\rbar$ is split at $p$ and $k(\rbar)=2$.
\end{enumerate}
We will then consider Conjecture \ref{conj:rhobarghost} for each twist $\rhobar = \rbar \otimes \omega^t$. Below are forms $f$ whose corresponding $r:=r_f$ account, respectively, for each of the above fives cases.
\begin{enumerate}
\item \label{item:first} $p=13$, $N=1$, and $f=\Delta$.
\item $p=7$, $N=11$, and $f$ corresponds to the  elliptic curve $X_0(11)$. 
\item $p=11$, $N=1$, and $f=\Delta$.
\item $p=23$, $N=1$, and $f=\Delta$.
\item $p=7$, $N=27$ and $f$ corresponds to the CM elliptic curve \cite[\href{http://www.lmfdb.org/EllipticCurve/Q/27.a1}{Elliptic Curve 27.a1}]{LMFDB}.
\end{enumerate}
We also considered two examples arising where $p$ is not $\Gamma_0(N)$-regular, but $\rbar$ still satisfies \eqref{BR}.  Both of these examples correspond to case (\ref{item:generic}) above:\ $\rbar$ is non-split locally at $p$ with Serre weight strictly between 2 and $p-1$.
\begin{enumerate}
\setcounter{enumi}{5}
\item $p=59$, $N=1$, and $f=\Delta$.
\item $p=19$, $N=3$, and $f$ is the unique form of weight 6 and level $\Gamma_0(3)$.
\end{enumerate}
We further considered some odd weight examples. The first two examples correspond to case (\ref{item:generic}) while in the third example $\rbar$ corresponds to case (\ref{item:d}):\ split at $p$ with Serre weight 3.
\begin{enumerate}
\setcounter{enumi}{7}
\item $p=23$, $N=3$, and $f$ is the weight 11 form corresponding to \cite[\href{http://www.lmfdb.org/ModularForm/GL2/Q/holomorphic/3/11/2/a/}{Newform 3.11.2.a}]{LMFDB}. This form is defined over $\Q(\sqrt{-5})$ and we view it as a form over $\Q_{23}$ under the embedding of $\Q(\sqrt{-5})$ into $\overline{\Q}_{23}$ which has $a_2(f) \equiv 4 \bmod{23}$.
\item The same as in example (h) except we choose the embedding which has $a_2(f) \congruent 19 \bmod 23$. 
\item 
$p=11$, $N=7$, and $f$ is the unique weight 3 cuspform which corresponds to \cite[\href{http://www.lmfdb.com/ModularForm/GL2/Q/holomorphic/7/3/6/a/}{Newform 7.3.6.a}]{LMFDB}. 
\end{enumerate}
Even though the $\rhobar$-dimension formulas from Section \ref{subsec:rhobar-spaces} assumed $p>3$, we nonetheless applied these same formulas to a few examples with $p=3$.  Both of these examples have $\rbar$ non-split at $p$ (cases (\ref{item:b}) and (\ref{item:c}) respectively).
\begin{enumerate}
\setcounter{enumi}{10}
\item 
$p=3$, $N=11$, and $f$ is the unique weight 2 cuspform which corresponds to $X_0(11)$.
\item 
$p=3$, $N=7$, and $f$ is the unique weight 4 cuspform. See \cite[\href{http://www.lmfdb.com/ModularForm/GL2/Q/holomorphic/7/4/1/a/}{Newform 7.4.1.a}]{LMFDB}. 
\end{enumerate}
Lastly, we considered an example with $\rbar$ globally reducible.  
\begin{enumerate}
\setcounter{enumi}{12}
\item \label{item:last} $p=17$, $N=2$, and $f = E_8$ so that $\rbar = \omega^7 \oplus 1$.
\end{enumerate}

Now that we have our fixed $\rbar$ we can state our numerical evidence. If $d\geq 0$ is an integer, write $G^{\leq d}_{\rhobar}$ for the truncation of $G_{\rhobar}$ in degree at most $d$.  
Via computer computations performed in Sage (\cite{sagemath}) and Magma (\cite{MagmaCite}) we verified the following:\

\begin{fact}
Let $\rbar = \rbar_f$ denote the mod $p$ Galois representation corresponding to any of the examples (\ref{item:first}) - (\ref{item:last}) above and let $M_{\rbar}$ denote the corresponding constant in the table below.  Then for each $2 \leq k \leq M_{\rbar}$ and $k \congruent k(\rbar) + 2t \bmod p-1$, we have
\begin{equation}
\label{eqn:conj}
\NP(\chr(U_p | S_k(\Gamma_0)_{\rbar \otimes \omega^t})) =
\NP(G^{\leq d}_{\rbar \otimes \omega^t}(k))
\end{equation}
where $d=\dim S_k(\Gamma_0)_{\rbar \otimes \omega^t}$. 

\begin{small}
\begin{center}
\begin{tabular}{|c|c|c|c|c|c|c|c|c|c|c|c|c|c|}
\hline
 & (a) & (b) &(c) &(d) &(e) &(f) &(g) &(h) &(i) &(j) &(k) &(l) &(m) \\
\hline
$M_{\rbar}$ & $1000$ & $200$ & $1000$ & $1000$ & $126$ & $1000$ & $450$ & $400$ & $400$ & $200$ & $200$ & $250$ & $380$ \\
\hline
\end{tabular}
\end{center}
\end{small}
\end{fact}
We note that while it looks like weights up to $1000$ are tested in example (a), say, half of them are the wrong parity and we are actually testing $p-1 = 12$ different conjectures. So each case of Conjecture \ref{conj:rhobarghost} was tested in (a) for approximately $500/12 \approx 41$ different weights.

We also note that while the right-hand side of \eqref{eqn:conj} is easy to compute, the left-hand side becomes very difficult to compute as $k$ grows. Indeed, for a fixed $k$ around $M_{\rbar}$, computing the slopes on the left-hand side of \eqref{eqn:conj} could take over a day of CPU time whereas computing the corresponding ghost slopes takes only a few seconds.

We close with two more observations.

\begin{remark}
\leavevmode
\begin{enumerate}
\item In examples (h) and (i), the associated $\rbar$'s are globally distinct, but locally isomorphic at $p$.  Moreover, each $\rbar$ occurs with multiplicity one in its Serre weight.  These facts force the associated ghost series in these two cases to be identical. In particular, Conjecture \ref{conj:rhobarghost} predicts that the slopes in these two cases are identical, and the data we collected above was in complete agreement with these observations.

\item  
We also performed numerical experimentations of the $w$-adic ghost slopes (see Remark \ref{rmk:wadic}) and we noticed that in the generic case all of the $w$-adic slopes were distinct.  By generic, we mean $\rbar$ is irreducible, locally reducible non-split at $p$, $2 < k(\rbar) < p-1$, and $\dim S_{k(\rbar)}(\Gamma)_{\rbar} =1$.
These computations were carried out for all $p<40$ and all generic values of $k(\rbar)$. We strongly suspect that this patterns holds in general as we not only observed that the slopes were distinct but could write down explicit formulas for the consecutive differences of the $w$-adic slopes in terms of $p$ and $k(\rbar)$ which were visibly non-zero when $2 < k(\rbar) < p-1$.

We note that this observation combined with Conjecture \ref{conj:rhobarghost} implies that the $\rhobar$-eigencurve is smooth in the spectral halo region, and, moreover, is just an infinite union of annuli.  One should compare this prediction with \cite[Conjecture 1.2]{LiuXiaoWan-IntegralEigencurves} where it is conjectured that the eigencurve is an infinite union of finite flat coverings of weight space.
\end{enumerate}
\end{remark}

\bibliography{ghost_bib}
\bibliographystyle{abbrv}

\end{document}